\documentclass{birkmult}

\usepackage{afterpage,float,amsmath}
\usepackage{longtable}
\usepackage{graphicx}
\usepackage{enumitem}

 \newtheorem{thm}{Theorem}
 \newtheorem{cor}[thm]{Corollary}
 \newtheorem{lem}[thm]{Lemma}
 \newtheorem{prop}[thm]{Proposition}
 \newtheorem{defn}[thm]{Definition}
 \newtheorem{rem}{Remark}
 
 \numberwithin{equation}{section}

\begin{document}

\title[Unitarization of non-primitive posets]
 {Unitarization of linear representations of non-primitive posets }

\author{Roman Grushevoi}

\address{%
Institute of Mathematics  NAS of Ukraine\\
Tereschenkivska str. 3\\
01601 Kiev\\
Ukraine}

\email{grushevoi@imath.kiev.ua}
\thanks{This work has been partially supported by the Scientific Program of National Academy
of Sciences of Ukraine, Project No~0107U002333. }

\author{Kostyantyn Yusenko}
\address{%
Institute of Mathematics  NAS of Ukraine\\
Tereschenkivska str. 3\\
01601 Kiev\\
Ukraine}

\email{kay@imath.kiev.ua}


\date{April 11, 2008}

 \maketitle

\section*{0. Introduction}

The representation theory of partially ordered sets (posets) in
linear vector spaces has been studied extensively and  found to be
of great importance for studying indecomposable representations of
group and algebras, Cohen-Macaulay modules and many others
algebraical objects (see \cite{Drozd, GabrielRoiter, Kleiner1,
Kleiner2, Simson1992} and many others). A representation of a given
poset $\mathcal P$ in some vector space $V$ is a collection
$(V;V_i),\ i\in \mathcal P$ of vector subspaces $V_i\subset V$  such
that $V_i\subset V_j$ as soon as $i\prec j$ in $\mathcal P$. Usually
such representations are studied up to equivalence (which is given
by linear bijections between two spaces that bijectively map the
corresponded subspaces). M. Kleiner and L. Nazarova (see
\cite{Kleiner1,Nazarova}) completely classified all posets into
three classes: finite type posets (posets that have finite number of
indecomposable nonequivalent representation), tame posets (posets
that have at most one-parametric family of indecomposable
representations in each dimension) and wild posets (the
classification problem of their indecomposable representations
contains as a subproblem a problem of classification up to conjugacy
classes a pair of two matrices).

It is also possible to develop a similar theory over Hilbert spaces.
By representation we understand a collection $(H;H_i)$ of Hilbert
subspaces in some Hilbert space $H$ such that $H_i\subset H_j$ as
soon as $i\prec j$. The equivalence between two system of Hilbert
subspaces is given by unitary operator which bijectively maps
corresponding subspaces. It turns out that in this case the
classification problem becomes much more harder: even the poset
$\mathcal P=\{a,b_1,b_2\}, b_1\prec b_2$ becomes a $*$-wild poset
(it is impossible to classify all representation of this poset in a
reasonable way see~\cite{KruglyaSamoilenko2}). We add an "extra"
relation
\begin{equation} \label{orthoscalarity}
    \alpha_1P_1+\ldots+\alpha_nP_n=\gamma I,
\end{equation}
between the projections $P_i:H\mapsto H_i$ on corresponding
subspaces for some weight $\chi=(\alpha_1,\ldots,\alpha_n) \in
\mathbb R^n_+$ (this relation will be called \emph{orthoscalarity
condition}). When the system of subspaces is a so-called
$m$-filtration, this relation plays an important role in different
areas of mathematics (see \cite{Totaro1994, Klyachko} and references
therein) and this is actually one of the original motivations to
investigate such representations of posets in Hilbert space.

The interconnection between linear and Hilbert representations of
the posets is given by \emph{unitarization} which asks whether for
given linear representation $(V,V_i)$ it is possible to provide a
hermitian structure in $V$ so that the linear relation
(\ref{orthoscalarity}) holds for some weight $\chi$. In
\cite{GrushevoyYusenko}  for the case when $\mathcal P$ is a
\emph{primitive} poset we proved that a poset $\mathcal P$ is of
finite orthoscalar type (has finitely many irreducible
representations with orthoscalarity condition up to the unitary
equivalence) if and only if it is of finite (linear) type. Also
there were proved that each indecomposable representation of poset
of finite (linear) type could be unitarized with some weight and for
each representation we described all appropriated for unitarization
weights.

In this paper we will prove the same for \emph{non-primitive}
posets. The approach given in \cite{GrushevoyYusenko} does not work
longer for \emph{non-pirmitive} case. We will use instead the notion
of $\chi$-stable representation from Geometric Invariant Theory of
the product of the grassmannians of $V$ (see \cite{Hu2004,
Totaro1994} and references therein). It turns out that for
indecomposable representation $\chi$-unitarizability is essentially
the same as $\chi$-stability. This gives a machinery to compute all
appropriated for unitarization weights.

The main results of the paper are the following theorems.

\begin{thm} \label {mainthmFin}
A partially ordered set $\mathcal P$ has finite number of
irreducible finite-dimensional Hilbert representations with
orthoscalarity condition if and only if it does not contain subsets
of the following form $(1,1,1,1)$, $(2,2,2)$, $(1,3,3)$, $(1,2,5)$,
$(N,4)$, where $(n_1,\ldots,n_s)$ denotes the cardinal sum of
linearly ordered set $\mathcal L_1,\ldots,\mathcal L_s$, whose
orders equal $n_1,\ldots,n_s$, respectively, and $(N,4)$ is the set
$\{a_1, a_2, b_1, b_2, c_1, c_2, c_3, c_4\}$, with the order $a_1
\prec a_2, b_1 \prec b_2, b_1 \prec a_2, c_1 \prec c_2 \prec c_3
\prec c_4$, and no other elements are comparable.
\end{thm}

\begin{thm} \label{mainthmUnit}
   Each indecomposable linear representation of the poset of finite type
can be unitarized with some weight.
\end{thm}

\textbf{Acknowledgments.} We would like to thank Prof. Yu.S.
Samoilenko and Prof. V.L. Ostrovskii for the statement of the
problem and Prof. Ludmila Turowska for stimulating discussion. The
second author thanks Chalmers University of Technology for
hospitality and for the fruitful environment and Thorsten Weist for
helpful remarks. The second author was partially supported by the
Swedish Institute and by Ukrainian Grant for Young Scientists.

\section{Preliminaries.}

In this section we will briefly recall some basic facts concerning
partially ordered sets, their representations and unitarization of
linear representations.


\subsection {Posets and Hasse quivers}

Let $(\mathcal P,\prec)$ be a finite partially ordered set (or poset
for short) which for us will be $\{a_1,\ldots,a_n\}$. By the width
of the poset $\mathcal P$ we understand the cardinality of the
largest antichain of $\mathcal P$, i.e. the cardinality of a subset
of $\mathcal P$ where any two element are incomparable.

A poset $\mathcal P$ of the width $s$ is called \emph{primitive} and
denoted by $(n_1,\ldots,n_s)$ if this poset is the cardinal sum of
$s$ linearly ordered sets $\mathcal L_1,\ldots,\mathcal L_s$ of
orders $n_1,\ldots,n_s$. Otherwise the poset is called
\emph{non-primitive}.

We will use the standard graphic representations for the poset $\mathcal{P}$
called  \emph{Hasse quiver}. This representation associates to each
elements $x \in \mathcal P$
a vertex $x$ 
and a unique arrow $x\rightarrow y,\ y\in \mathcal P$ if $x\prec y$
and if there is no $z \in \mathcal P$ such that $x\prec z \prec y$.
For example, let $\mathcal P= (N,2)=\{a_1,a_2,b_1,b_2,c_1,c_2\}$
with the following order
$$
    a_1\prec a_2, \quad b_1\prec b_2, \quad c_1\prec c_2, \quad b_1\prec a_2,
$$
then the corresponding Hasse quiver is the following:
%
\begin{center}
\begin{picture}(50,60)
      \put(0,15){$a_1$}
      \put(20,15){$b_1$}
      \put(40,15){$c_1$}
      \put(2,24){\vector(0,1){15}}
      \put(22,24){\vector(0,1){15}}
      \put(42,24){\vector(0,1){15}}
      \put(0,44){$a_2$}
      \put(20,44){$b_2$}
      \put(40,44){$c_2$}
      \put(19,24){\vector(-1,1){15}}
      \put(50,25){$.$}
\end{picture}
\end{center}

\subsection {Linear representations of posets. Indecomposability and Bricks}

\noindent By a linear representation $\pi$ of a given poset
$\mathcal P$ in a complex vector space $V$ we understand a rule that
to each element $i \in \mathcal P$ associates a subspace $V_i
\subseteq V$ in such a way that $i \prec j$ implies $V_i \subseteq
V_j$. We will often think of $\mathcal P$ as a set
$\{1,2,\ldots,n\}$, where $n$ is the cardinality of $\mathcal P$ and
write $\pi=(V;V_1,\ldots,V_n)$ or $\pi=(V;V_i),$ notation
$\pi(i)=V_i$, $\pi_0=V$ also will be used.

 By the dimension vector $d_\pi$ of the representation $\pi$ we
understand a vector $d_\pi=(d_0;d_1,\ldots,d_n)$, where  $d_0=\dim
V$, $d_i=\dim (V_i/\sum_{j\prec i} V_j$). There is qudratic form
$Q_{\mathcal P}$ on $\mathbb Z^{card(\mathcal P)+1}$ given by
$$
    Q_{\mathcal P}(x_0, x_1,\ldots, x_n) =x_0^2+\sum_{i\in \mathcal P} x_i^2+ \sum_{a,b \in \mathcal{P}, a\prec b} x_a x_b - \sum_{a\in \mathcal P} x_0x_a.
$$

 Throughout the paper we  denote by $e_i$ the  $i$-th
coordinate vector $e_i=(\delta_{ij})$, by $e_{i_1\ldots i_k}$ we
understand the vector $e_{i_1}+\ldots+e_{i_k}$, and by $\langle
x_1,\ldots,x_n \rangle$ the complex vector space spanned by vectors
$x_1,\ldots,x_n \in V$. We will use a graphical picture for the
representation of posets. For example the following picture
describes the representation $\pi=(\mathbb C \langle e_1,e_2
\rangle;\langle e_1 \rangle,\langle e_2 \rangle,\langle e_{12}
\rangle)$ for the poset $(1,1,1)$
%
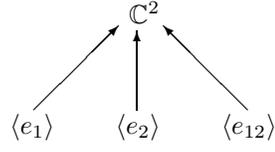
\begin{center}
\begin{picture}(100,70)
      \put(0,15){$\langle e_1 \rangle$}
      \put(40,15){$\langle e_2 \rangle$}
      \put(80,15){$\langle e_{12} \rangle$}
      \put(48,24){\vector(0,1){30}}
      \put(9,24){\vector(1,1){32}}
      \put(90,24){\vector(-1,1){32}}
      \put(45,57){$\mathbb C^2$}
\end{picture}
\end{center}

In fact the set of all linear representations of a poset $\mathcal
P$ forms the additive category $\rm{Rep}(\mathcal P)$, where the set
of morphisms $\rm{Mor}(\pi_1,\pi_2)$ between two representations
$\pi_1=(V;V_1,\ldots,V_n)$ and $\pi_2=(W;W_1,\ldots,W_n)$ consists
of linear maps $C:V \rightarrow W$, such that $C(V_i)\subset W_i$.
Two representations $\pi_1$ and $\pi_2$ of  $\mathcal P$ are
isomorphic (or equivalent) if there exists an invertible morphism $C
\in \rm{Mor}(\pi_1,\pi_2)$, i.e. there exist an invertible linear
map $C:V \rightarrow W$ such that $C(V_i)=W_i$.

 One can define a direct sum $\pi=\pi_1 \oplus \pi_2$ of two objects $\pi_1, \pi_2
\in \mathcal P$ in the following way: $$\pi=(V \oplus W; V_1 \oplus
W_1,\ldots,V_n \oplus W_n).$$ Using the notion of direct sum it is
natural to define \emph{indecomposable} representations as the
representations that are not isomorphic to the direct sum of two
non-zero representations, otherwise representations are called
\emph{decomposable}. It is easy to show that a representation $\pi$
is indecomposable if and only if there is no non-trivial idempotents
in endomorphism ring $\rm{End}(\pi)$. A representation $\pi$ is
called \emph{brick} if there is no non-trivial endomorphism of this
representation (or equivalently when the ring $\rm{End}(\pi)$ is
one-dimensional). It is obvious that if a representation is
\emph{brick} then it is \emph{indecomposable}. But there exist
\emph{indecomposable} representations of  posets which are not
\emph{brick}, for example
$$ A_\alpha= \left(
         \begin{array}{cc}
           1 & \alpha \\
           0 & 1 \\
         \end{array}
       \right), \quad \alpha \in \mathbb R \backslash \{0\},$$
the representation $\pi_{\alpha}=(V;V_1,V_2,V_3,V_4)$ of the poset
$\mathcal N=(1,1,1,1)$
\begin{align*}
    V=\mathbb C^2 \oplus \mathbb C^2; \quad V_1=\mathbb
    C^2 \oplus 0, \quad V_2=0 \oplus \mathbb C^2, \\
    V_3=\left\{ (x,x)\in \mathbb C^4\ |\ x \in \mathbb C^2\right \}, \quad
    V_4=\left\{ (x,A_\alpha x)\in \mathbb C^4\ |\ x \in \mathbb C^2\right \},
\end{align*}
is indecomposable but is not brick.

Recall that a poset $\mathcal P$ is called a poset of \emph{finite
(linear) type} if there exist only finitely many non-isomorphic
indecomposable representation of $\mathcal P$ in the category
$\rm{Rep}(\mathcal P)$. Result obtained by M.M.Kleiner
\cite{Kleiner1}  gives a complete description of the posets of
finite type.

\begin{thm}(see \cite{Kleiner1}, Theorem 1)
   A poset $\mathcal P$ is a set of finite type if and only if
it does not contain subsets of the form $(1,1,1,1)$, $(2,2,2)$,
$(1,3,3)$, $(1,2,5)$ and $(N,4)$, where
\begin{gather*}
    (N,4)=\{a_1,a_2,b_1,b_2,c_1,c_2,c_3,c_4\}, \\ a_1\prec
a_2, \ b_1\prec b_2,\ b_1\prec a_2,\ c_1\prec c_2 \prec c_3 \prec
c_4.
\end{gather*}
\end{thm}

\begin{rem}
The posets in previous theorem will be called critical henceforth.
\end{rem}

Kleiner described also all indecomposable representations of posets
of finite type up to equivalence using the notion of \emph{sincere}
representations (see~\cite{Kleiner2}).

\begin{defn}  We call a representation $\pi$ \emph{sincere}
if it is indecomposable and the components of the dimension vector
$d_\pi=(d_0;d_1,\ldots,d_n)$ satisfy $d_i\not=0;$ otherwise we say
that the representation is \emph{degenerated}.
\end{defn}

A poset is called \emph{sincere} if it has at least one
\emph{sincere} representation. It is easy to see that any
indecomposable representation of a poset of finite type actually is
a sincere representation of its some sincere subposet. To describe
all indecomposable representations of a fixed poset $\mathcal P$ of
finite type one needs to describe all sincere representations of its
all sincere subposets $\mathcal R$ (including itself if $\mathcal P$
is sincere) and to add zero spaces $V_i/\sum_{j\prec i} V_j$ if
$i\not\in \mathcal R$.

\subsection{Unitary representations of posets}

In the spirit of a number of previous articles  we study
representation theory of posets over Hilbert spaces. Denote by
$\rm{Rep}(\mathcal P, H)$ a sub-category in $\rm{Rep}(\mathcal P)$,
defined as follows: its set of objects consists of
finite-dimensional Hilbert spaces and two objects $\pi=(H,H_i)$ and
$\tilde{\pi}=(\tilde H,\tilde H_i)$ are equivalent in
$\rm{Rep}(\mathcal P, H)$ if there exists a unitary operator
$U:H\rightarrow\tilde H$ such that $U(H_i)= \tilde H_i$ (unitary
equivalent). Representation $\pi\in\rm{Rep}(\mathcal P, H)$ is
called \emph{irreducible} iff the $C^*$-algebra generated by set of
orthogonal projections $\{P_i\}$ on the subspaces $\{H_i\}$ is
irreducible. Let us remark that indecomposability of a
representation $\pi$ in $\rm{Rep}(\mathcal P)$ implies
irreducibility of $C^*(\{P_i, i\in\mathcal P\})$ but the converse is
false.

The problem of classification  all irreducible objects in the
category $\rm{Rep}(\mathcal P, H)$ becomes much harder. Even for the
primitive poset $\mathcal P=(1,2)$ it is hopeless to describe in a
reasonable way all its irreducible representations: indeed this lead
us to classify up to unitary equivalence three subspace in a Hilbert
space, two of which are orthogonal, but it is well-known due
to~\cite{KruglyaSamoilenko2} that such problem is $*$-wild. Hence it
is natural to consider some additional relation.

Let us consider those objects $\pi \in \rm{Rep}(\mathcal P, H)$,
$\pi=(H;H_1,\ldots,H_n)$, for which  the following linear relation
holds:
\begin{equation}
    \alpha_1P_1+\ldots+\alpha_nP_n=\gamma I, \label{linear_rel}
\end{equation}
where $\alpha_i,\gamma$ are some positive real numbers, and $P_i$
are the orhoprojections on the subspaces $H_i$. These objects form a
category which will be also denoted by $\rm{Rep}(\mathcal P, H)$.
Such representations will be called \emph{orthoscalar
representations}.


\subsection{Unitarization}

Obviously there exists a forgetful functor from $\rm{Rep}(\mathcal
P,H)$ to $\rm{Rep}(\mathcal P)$ which maps each system of Hilbert
spaces to its underlying system of vector spaces. We ask whether
there exists "functor in reverse direction''?

\begin{defn}
We say that a given representation $\pi \in \rm{Rep}(\mathcal P)$,
$\pi=(V;V_1,\ldots,V_n)$ of the poset $\mathcal P$ can be unitarized
with a weight (or is unitarizable) $\chi=(\alpha_1,\ldots,\alpha_n)$
if it is possible to choose hermitian structure
$\langle\cdot,\cdot\rangle_{\mathbb C}$ in $V$, so that the
corresponding projections $P_i$ onto subspace $V_i$ satisfy the
following relation:
$$
    \alpha_1P_1+\ldots+\alpha_nP_n=\lambda_\chi(\pi) I,
$$
where $\lambda_\chi(\pi)$ is equal to $\frac{1}{\dim
V}\sum_{i=1}^n\alpha_i \dim V_i$.
\end{defn}

For a given linear representation $\pi$ of the poset $\mathcal P$ by
$\triangle_\pi^{\mathcal P}$ we denote the set of those weights
$\chi$ that are appropriated for unitarization. And correspondingly
we say that representation $\pi$ can be \emph{unitarized} if the set
$\triangle_\pi^{\mathcal P}$ is nonempty.

In \cite{GrushevoyYusenko} we showed that each indecomposable
non-degenerated representation of primitive poset of finite type can
be unitarized and for each such representation we completely
described the sets $\triangle_\pi^{\mathcal P}$. The approach
provided in \cite{GrushevoyYusenko} does not work longer for
\emph{non-pirmitive} case. In this paper we will use rather
different approach which comes from Geometric Invariant Theory and
gives exact criteria for unitarization.

\section{Balanced metric and stable representations of posets}

Let $\pi=(V;V_1,\ldots,V_n)$ be a system of subspaces (in particular
it can be a representation of some poset $\mathcal P$) in a complex
vector space $V$ and let $\chi=(\alpha_1,\ldots,\alpha_n)$ be some
weight, i.e. the vector from $\mathbb R^n_+$. Denote by
$\lambda_\chi(\pi)$ the number defined by
$$
    \lambda_\chi(\pi)=\frac{1}{\dim V}\sum_{i=1}^n \alpha_i \dim V_i.
$$
If $U$ is a subspace of $V$ one can form another system of subspaces
$\pi \cap U$ generated by $\pi$ and $U$
$$
    \pi \cap U=(U;V_1\cap U,\ldots,V_n \cap U).
$$
By $\lambda_\chi(\pi \cap U)$ we will understand the number given by
$$
    \lambda_\chi(\pi \cap U)=\frac{1}{\dim U}\sum_{i=1}^n \alpha_i \dim (V_i\cap U).
$$

\begin{defn}
    We say that a system of subspace $\pi=(V;V_1,...,V_n)$ is
$\chi$-stable if for each proper subspace $U\subset V$ the following
inequality holds
    $$
        \lambda_\chi(\pi\cap U)<\lambda_\chi(\pi).
    $$
\end{defn}

Suppose for a moment that for a system $\pi=(V;V_1,\ldots,V_n)$ we
have chosen a sesquilinear form $\langle\cdot,\cdot\rangle$ on $V$
so that
$$
    \chi_1 P_{V_1}+\ldots+\chi_n P_{V_n}=\lambda_\chi(\pi)I_V,
$$ following \cite{Hu2004} we will call such form
\emph{$\chi$-balanced} metric. For a system $\pi$ to possess a
$\chi$-balanced metric is essentially the same as to be unitarized
with the weight $\chi$. The list of necessarily restrictions on
weight $\chi$ can be obtained by the following lemma
\begin{lem}
    If the indecomposable system of subspaces $\pi$ possesses $\chi$-balanced metric then $\pi$ is $\chi$-stable.
\end{lem}

\begin{proof}
The proof of these statement can be obtained by taking the trace
from linear equation. Indeed let $\pi=(V;V_1,\ldots,V_n)$ possesses
a $\chi$-balanced metric, i.e. a metric $\langle\cdot,\cdot\rangle$
on $V$ such that
\begin{equation} \label{chi-m}
    \alpha_1 P_{V_1}+\ldots+\alpha_n P_{V_n}=\lambda_\chi(\pi)I_V.
\end{equation}
By taking the trace from the left and right hand sides of
(\ref{chi-m}) we can write  $\lambda_\chi(\pi)$ as
$$
\lambda_\chi(\pi)=\frac{tr(\alpha_1 P_{V_1}+\ldots+\alpha_n
P_{V_n})}{tr(I_V)}=\frac{1}{tr(I_V)}\sum_{i=1}^n\alpha_i
tr(P_{V_i}).
$$
Let $U$ be some proper subspace of $V$. Denote by $P_U$ an
orthogonal projection on $U$. Multiplying (\ref{chi-m}) by $P_U$
from the left we obtain
$$
    \chi_1 P_{V_1}P_U+\ldots+\chi_n P_{V_n}P_U=\lambda_\chi(\pi)P_U,
$$
then taking trace of the last we get
$$
    \lambda_\chi(\pi)=\frac{1}{tr(P_U)}\sum_{i=1}^n\alpha_i tr(P_{V_i}P_U).
$$
Observe that $tr(P_{V_i}P_U)\geq tr(P_{V_i\cap U})$ (this can be
proved using spectral theorem for the pair of projection and then by
restriction to two-dimensional representation). This gives us
$$
   \lambda_\chi(\pi)=\frac{1}{tr(P_U)}\sum_{i=1}^n\alpha_i tr(P_{V_i}P_U)\geq \frac{1}{tr(P_U)}\sum_{i=1}^n\alpha_i tr(P_{V_i\cap
U})=\lambda_\chi(\pi\cap U).
$$
It remains to prove that the inequality is strict. Indeed assume
that $tr(P_{V_i}P_U)=tr(P_{V_i\cap U})$ for all $i$. Then it is easy
to see the all $P_{V_i}$ commutes with $P_U$ (again using spectral
theorem for the pair of two projections) hence the subspace $U$ is
invariant with respect to projections $P_i$ which means that the
representation $\pi$ is decomposable. Therefore
$\lambda_\chi(\pi\cap U) < \lambda_\chi(\pi)$ for all proper
subspaces $U$, and hence $\pi$ is $\chi$-stable.
\end{proof}

A natural question arises whether the reverse statement is true,
i.e. does every $\chi$-stable system $\pi$ possesses $\chi$-balanced
metric?

When $\pi$ is a collection of filtrations (recall that a filtration
is a chain of subspaces $V_0\subset\ldots\subset V_m )$ this
assertion was proved by Totaro (\cite{Totaro1994}) and Klyachko
(\cite{Klyachko}). If fact this can be proved for any configuration
of subspaces and any weight $\chi$. Here we will reproduce shortly
what was done in \cite{Hu2004}.

Let $V$ be a complex vector space and let $\chi \in \mathbb N^n$.
Consider the product of Grassmanians
$$
    \textrm{Gr}(k_1,V)\times\ldots\times\textrm{Gr}(k_n,V).
$$
Any system of subspaces $\pi=(V;V_1,\ldots,V_n)$ of vector space $V$
with dimension vector equal to $d=(\dim V;k_1,\ldots,k_n)$ can be
considered as a point of $\prod_{i=1}^{n}\textrm{Gr}(k_i,V)$. We
equip $\prod_{i=1}^{n}\textrm{Gr}(k_i,V)$ with simplectic form
$\delta$, which is the skew bilinear form
\begin{equation*}
\begin{split}
\delta&:\prod_{i=1}^{n}\textrm{Gr}(k_i,V) \times
\prod_{i=1}^{n}\textrm{Gr}(k_i,V) \rightarrow \mathbb C \\
    \delta&:(\pi,\tilde \pi)\mapsto\sum_i \chi_i tr(A_i \tilde A_i^*),
\end{split}
\end{equation*}
where $A_i$ and $\tilde A_i$ is a matrix representation of $V_i$ and
$\tilde V_i$ (their columns form an orthonormal bases for $V_i$ and
$\tilde V_i$ correspondingly), and $*$ is adjoint correspondingly to
standart hermitian metric $\langle\cdot,\cdot\rangle$ on $V$.

As Lie group $SU(V)$ acts diagonally on
$\prod_{i=1}^{m}\textrm{Gr}(k_i,V)$ preserving symplectic form
$\sigma$, the action is given by operating on $V$ (via its linear
representation). The corresponding moment map
$\Phi:\prod_{i=1}^{m}\textrm{Gr}(k_i,V) \rightarrow su^*(V)$.
$su^*(V)$ the dual of Lie algebra of $SU(V)$ which is given by the
algebra of traceless Hermitian matrices over $V$. This moment map is
given by
$$
    \Phi(\pi)=\sum_i\chi_i A_iA_i^*-\lambda_\chi(\pi)I,
$$

Assuming that $\pi$ is $\chi$-stable  is possible to find (see
\cite{Hu2004}) such $g \in SL(V)$ that $\Phi(g\cdot\pi)=0$. Then
correspondently to hermitian metric
$g\langle\cdot,\cdot\rangle=\langle g \cdot, g \cdot\rangle$ the
following holds
$$
    \sum_i \chi_i A_i A_i^*=\lambda_\omega(\pi) I.
$$
Taking into account that $A_iA_i^*$ is an orthogonal projection on
$V_i$ correspondinly to $g\langle\cdot,\cdot\rangle$ we get
desirable result.

To conclude it remains to note that for  $\chi \in \mathbb R^n_+$
one can find appropriated sequence of rational $\chi_n$ that tends
to $\chi$ and that $\pi$ is $\chi_n$-stable (it is possible because
stable condition is open) and then one can make use (for example)
Shulman's lemma about representation of limit relation.

Summing up the following theorem holds (\cite{Hu2004}).

\begin{thm} Let $\chi$ be the weight. For the indecomposable system of subspaces $\pi$
the following conditions are equivalent:
\begin{enumerate}[label=(\roman{*})]
\item $\pi$  can be unitarized with weight $\chi$;
\item $\pi$ is $\chi$-stable system.
\end{enumerate}

\end{thm}
This theorem gives exact criteria of unitarization of any linear
indecomposable representation of partially ordered set with the
weight $\chi$. On practice in order to check the $\chi$-stability
for a given representation $\pi=(V;V_1,\ldots,V_n)$ of some poset
with dimension vector $d=(\dim V_0;\dim V_1,\ldots,\dim V_n)$ one
can describe all possible subdimension vectors $d^\prime=(\dim
U;\dim(V_1\cap U),\ldots,\dim(V_n\cap U))$, $U\subset V$ and to
check for these vector stability condition
$$
    \frac{1}{d^\prime_0}\sum\chi_i d^\prime_i<\frac{1}{d_0}\sum\chi_i
    d_i,
$$
Let us remark that on subdimension vectors there exist a natural
coordinate partial order. That is evident that to check the
stability condition one should check inequality above for maximal
vectors.

\section{Proof of the Theorem~\ref{mainthmFin}}

\

\emph{Sufficiently.} Let $\mathcal P$ be the poset which does not
contain any of critical posets. Assume that these posets could have
infinite number of unitary inequivalent Hilbert space
representations with fixed weight. If two Hilbert space
representations with the same weight are unitary non-equivalent then
they  are linearly inequivalent (due to
~\cite{KruglyaNazarovaRoiter}, Theorem 1). Hence each such poset has
infinite number of indecomposable linear representation. But this
contradicts Kleiner's theorem.

\emph{Necessity.} Our aim is for each critical poset $\mathcal P$ to
build infinite series of indecomposable pairwise nonequivalent
Hilbert representations of this poset with the same weight. For the
primitive case this can be done using the connection between Hilbert
orthoscalar representations of the posets with the representations
of some certain class of $*$-algebras that connected with
star-shaped graphs (see for example \cite{OstrovskyiSamoilenko}).
But this approach does not work for the nonprimitive case (namely
for the set $(N,4)$). Here we consider quite different approach
which is based on unitarization.

Let $\mathcal P$ be a poset and let $I \subset \mathcal P$ ($I$ can
be empty). Define an extended poset $\mathcal {\tilde P}_I$ by
adding to $\mathcal P$ an element $\tilde p$ subject to the
relations $\tilde p \prec i$, $i\in I$. Let $\pi=(V;V_i)$ be a
representation of the poset $\mathcal P$. Assume that there are two
linearly independent vectors $v_1,v_2\in V$ such that the following
conditions are satisfied
$$
\dim((\sum_{i\in I} V_i + \langle v_1+\lambda v_2\rangle) \cap
(\sum_{i\in I} V_i + \langle v_1+\tilde \lambda
v_2\rangle))=\dim(\sum_{i\in I} V_i),\quad
$$
where $\lambda\neq\tilde\lambda \in D$, and $D$ is dense in $\mathbb
C$. One can show that such vectors exist if $\dim(\sum_{i\in I}
V_i)+1<\dim(V)$. We can define a family of representation
$\tilde\pi_{\lambda}$ of the $\mathcal {\tilde P}_{I}$, by letting
$$
    \tilde \pi_\lambda(x)=\left\{
                    \begin{array}{c}
                      \pi(x), \quad x\neq\tilde p, \\
                      \sum_{i\in I}\pi(i)+ \langle v_1+\lambda v_2\rangle, \quad x=\tilde p. \\
                    \end{array}
                  \right.
$$

The following proposition is straightforward.

\begin{prop}
    Assume that $\pi=(V;V_i)$ is a brick representation of the poset $\mathcal P$ and $\tilde \pi_\lambda(\tilde p)$
    is the corresponding family of representations of $\mathcal P_I$ for appropriated choosen $v_1,v_2 \in V$. Then the following holds
    \begin{enumerate}
        \item The representation $\tilde\pi _\lambda$ is brick ($End(\tilde\pi_\lambda)\cong\mathbb C$) for
        each $\lambda \in \mathbb C$.

        \item If $\lambda\neq\lambda^\prime$ then $\tilde\pi_\lambda$ is not equivalent to $\tilde\pi_{\lambda^\prime}$.

        \item If $\lambda \neq
        \lambda^\prime$ and $\pi_\lambda$, $\pi_{\lambda^\prime}$ are $\chi$-stable  then the corresponding systems of projection
        (after unitarization) are unitary inequivalent.
    \end{enumerate}
\end{prop}

Using the construction above for each pre-critical poset (that is
critical poset without one element) we will build its extended poset
that coincide with critical poset and we will choose appropriated
representation of pre-critical what allows to build extended
representations of critical posets. These family of extended
representation are given in dimension $d^\mathcal P$ which is the
imaginary root of the corresponding quadratic form $Q_\mathcal{P}$,
$Q_\mathcal{P}(d^\mathcal P)=0$. We will prove that these
representations are stable for all $\lambda \in \mathbb C, \lambda
\neq0,1$ (see Appendix B for the description of subdimension
vectors) with the same weight $\chi^\mathcal{P}$ which is defined by
the dimension vector in the following way:
$\chi^\mathcal{P}_i=d^\mathcal{P}_i$ and
$\lambda_\chi=d^\mathcal{P}_0$.

\emph{1) Case $(1,1,1,1)$}.

For $\mathcal P=(1,1,1)$, $I=\emptyset$ the extended poset
$\mathcal{\tilde P}_I$ is equal to $(1,1,1,1)$ and for the
representation $\pi=(\mathbb C^2;\langle e_1\rangle;\langle
e_2\rangle; \langle e_1+e_2\rangle)$ its extended representation has
the following form.


\begin{center}
\begin{picture}(150,20)
      \put(0,10){$\langle e_1 \rangle$}
      \put(30,10){$\langle e_2 \rangle$}
      \put(60,10){$\langle e_1+e_2 \rangle$}
      \put(110,10){$\langle e_1+\lambda e_2 \rangle$}
\end{picture}
\end{center}

The dimension vector is equal to $d^{(1,1,1,1)}=(2;1;1;1;1)$ and the
representations are $(1,1,1,1)$--stable.

\emph{2) Case $(2,2,2)$.}

For $\mathcal P=(1,2,2)$, $I=\{a_1\}$ the extended poset
$\mathcal{\tilde P}_I$ is equal to $(2,2,2)$ and for the
representation $\pi=(\mathbb C^3;\langle e_{123}\rangle;\langle
e_2\rangle, \langle e_1,e_2\rangle;\langle e_2\rangle, \langle
e_2,e_3\rangle)$ its extended representation has the following form


\begin{center}
\begin{picture}(150,51)
      \put(30,5){$\langle e_{123} \rangle$}
      \put(97,5){$\langle e_1 \rangle$}
      \put(145,5){$\langle e_3 \rangle$}
      \put(42,16){\vector(0,1){25}}
      \put(105,16){\vector(0,1){25}}
      \put(153,16){\vector(0,1){25}}
      \put(10,45){$\langle e_{123}, e_1+\lambda e_3 \rangle$}
      \put(90,45){$\langle e_1, e_2 \rangle$}
      \put(138,45){$\langle e_2, e_3 \rangle$}
\end{picture}
\end{center}

The dimension vector is equal to $d^{(2,2,2)}=(3;1,1;1,1;1,1)$ and
the representations are $(1,1,1,1,1,1)$--stable.

\emph{3) Case $(1,3,3)$.}

For $\mathcal P=(1,2,3)$, $I=\{b_2\}$ the extended poset
$\mathcal{\tilde P}_I$ is equal to $(1,3,3)$ and for the
representation $\pi=(\mathbb C^4;\langle
e_{123},e_{24}\rangle;\langle e_4\rangle, \langle e_1,e_4\rangle;
\langle e_3\rangle,\langle e_2,e_3\rangle, \langle
e_1,e_2,e_3\rangle)$ its extended representation has the following
form

\begin{center}
\begin{picture}(150,91)
      \put(15,5){$\langle e_{123}, e_{24} \rangle$}
      \put(92,5){$\langle e_4 \rangle$}
      \put(155,5){$\langle e_3 \rangle$}
      \put(100,16){\vector(0,1){25}}
      \put(163,16){\vector(0,1){25}}
      \put(85,45){$\langle e_1, e_4 \rangle$}
      \put(148,45){$\langle e_2, e_3 \rangle$}
      \put(100,56){\vector(0,1){25}}
      \put(163,56){\vector(0,1){25}}
      \put(60,85){$\langle e_1, e_4, e_2+\lambda e_3 \rangle$}
      \put(143,85){$\langle e_1, e_2, e_3 \rangle$}
\end{picture}
\end{center}

Dimension vector is equal to $d^{(1,3,3)}=(4;2;1,1,1;1,1,1)$ and the
representations are $(2,1,1,1,1,1,1)$--stable.

\emph{4) Case $(1,2,5)$.}

For $\mathcal P=(1,2,4)$, $I=\{c_4\}$ the extended poset
$\mathcal{\tilde P}_I$ is equal to $(1,2,5)$ and for the
representation
\begin{multline*}
\pi=(\mathbb C^6;\langle e_{123},e_{245},e_{16}\rangle;\langle
e_5,e_6\rangle, \langle e_1,e_2,e_5,e_6\rangle;\\ \langle
e_4\rangle,\langle e_3,e_4\rangle, \langle
e_2,e_3,e_4\rangle,\langle e_1,e_2,e_3,e_4\rangle)
\end{multline*}
 its extended
representation has the following form

%

\begin{center}
\begin{picture}(150,155)
      \put(15,5){$\langle e_{123},e_{245},e_{16}\rangle$}
      \put(99,5){$\langle e_5, e_6 \rangle$}
      \put(165,5){$\langle e_4 \rangle$}
      \put(113,16){\vector(0,1){20}}
      \put(173,16){\vector(0,1){20}}
      \put(85,40){$\langle e_1, e_2, e_5, e_6 \rangle$}
      \put(158,40){$\langle e_3, e_4 \rangle$}
      \put(173,51){\vector(0,1){20}}
      \put(153,75){$\langle e_2, e_3, e_4 \rangle$}
      \put(173,86){\vector(0,1){20}}
      \put(145,110){$\langle e_1, e_2, e_3, e_4 \rangle$}
      \put(173,121){\vector(0,1){20}}
      \put(125,145){$\langle e_1, e_2, e_3, e_4, e_5+\lambda e_6 \rangle$}
\end{picture}
\end{center}

 In this case the dimension vector is equal to
$d^{(1,2,5)}=(6;3;2,2;1,1,1,1,1)$ and the representations are
$(3,2,2,1,1,1,1,1)$--stable.

\emph{5) Case $(N,4)$.}

For $\mathcal P=(1,2,4)$, $I=\{a_1,b_1\}$ its extended poset
$\mathcal{\tilde P}_I$ is equal to $(N,4)$ and for the
representation $$\pi=(\mathbb C^5;\langle
e_{235},e_{134}\rangle;\langle e_5\rangle, \langle
e_1,e_2,e_5\rangle;\langle e_4\rangle, \langle
e_3,e_4\rangle,\langle e_2,e_3,e_4\rangle,\langle
e_1,e_2,e_3,e_4\rangle)$$ its extended representations has the
following form


\begin{center}
\begin{picture}(150,120)
      \put(28,5){$\langle e_{235}, e_{134}\rangle$}
      \put(125,5){$\langle e_5 \rangle$}
      \put(180,5){$\langle e_4 \rangle$}
      \put(50,16){\vector(0,1){20}}
      \put(133,16){\vector(0,1){20}}
      \put(188,16){\vector(0,1){20}}
      \put(0,40){$\langle e_{235},e_{134}, e_5, e_3+\lambda e_4\rangle$}
      \put(112,40){$\langle e_1, e_2, e_5\rangle$}
      \put(173,40){$\langle e_3, e_4 \rangle$}
      \put(188,51){\vector(0,1){20}}
      \put(168,75){$\langle e_2, e_3, e_4 \rangle$}
      \put(188,86){\vector(0,1){20}}
      \put(160,110){$\langle e_1, e_2, e_3, e_4 \rangle$}
      \put(129,16){\vector(-4,1){75}}
\end{picture}
\end{center}

The dimension vector is equal to $d^{(\mathcal
N_4)}=(5;2,1;1,2;1,1,1,1)$ and the representations are
$(2,1,1,2,1,1,1,1)$--stable.

So, for each critical posets $\mathcal{ \tilde P}$ we build the
infinite family of pairwise inequivalent $d^{\mathcal P}$--stable
brick representation $\pi_\lambda$ in the dimension $d^{\mathcal P}$
(see appendix for the proof of stability). Hence each critical poset
has infinite number of pairwise unitary inequivalent representations
satysfying
$$
    d^{\mathcal P}_1 P_1+\ldots+d^{\mathcal P}_n P_n=d^{\mathcal P}_0
    I,
$$
i.e. each critical posets has infinite Hilbert space representable
type. This completes the proof of Theorem~\ref{mainthmFin}.


\begin{rem} Theorem 1 can be reformulated in the following way ---  a poset has finite orthoscalar Hilbert type if and only if
it has finite linear type.
\end{rem}

\section{Quite sincere representations and the proof of Theorem~\ref{mainthmUnit}}


\subsection{Quite sincere representations}

To describe all irreducible orthoscalar representations of posets of
finite type we need a new notion of sincere representation which we
will call quite sincerity.

The following is for linear representations of posets.

\begin{defn}  Let $\mathcal P=\{1,\ldots,n\}$ be a poset.
We call a representation $\pi$ of $\mathcal P$ \emph{quite sincere}
if it is indecomposable and the following conditions holds for all
$i=1,\dots,n$:
\begin{itemize}
  \item $\pi(i)\not=0;$
  \item $\pi(i)\not=\pi(0)$;
  \item $\pi(i)\not=\pi(j)$ as \ $i\prec j$.
\end{itemize}
\end{defn}

\begin{defn}  We say that the poset is \emph{quite sincere} if it
has at least one \emph{quite sincere} representation.
\end{defn}
The following theorem describes all quite sincere posets of finite
type and gives all their quite sincere representations.

\begin{thm} The set of quite sincere posets of finite type consists of four primitive posets $(1,1,1)$, $(1,2,2)$, $(1,2,3)$, $(1,2,4)$ and following non-primitive posets:
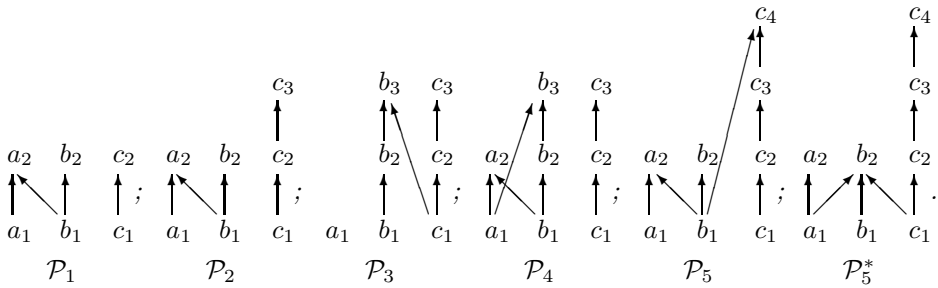
\begin{center}
\begin{picture}(50,60)
      \put(0,15){$a_1$}
      \put(20,15){$b_1$}
      \put(40,15){$c_1$}
      \put(2,24){\vector(0,1){15}}
      \put(22,24){\vector(0,1){15}}
      \put(42,24){\vector(0,1){15}}
      \put(0,44){$a_2$}
      \put(20,44){$b_2$}
      \put(40,44){$c_2$}
      \put(19,24){\vector(-1,1){15}}
      \put(48,30){;}
      \put(15,0){$\mathcal P_1$}
\end{picture}\quad
\begin{picture}(50,75)
      \put(0,15){$a_1$}
      \put(20,15){$b_1$}
      \put(40,15){$c_1$}
      \put(2,24){\vector(0,1){15}}
      \put(22,24){\vector(0,1){15}}
      \put(42,24){\vector(0,1){15}}
      \put(0,44){$a_2$}
      \put(20,44){$b_2$}
      \put(40,44){$c_2$}
      \put(42,52){\vector(0,1){15}}
      \put(40,71){$c_3$}
      \put(19,24){\vector(-1,1){15}}
      \put(48,30){;}
      \put(15,0){$\mathcal P_2$}
\end{picture}\quad
\begin{picture}(50,75)
      \put(0,15){$a_1$}
      \put(20,15){$b_1$}
      \put(40,15){$c_1$}
      \put(22,24){\vector(0,1){15}}
      \put(42,24){\vector(0,1){15}}
      \put(20,44){$b_2$}
      \put(40,44){$c_2$}
      \put(22,52){\vector(0,1){15}}
      \put(42,52){\vector(0,1){15}}
      \put(20,71){$b_3$}
      \put(40,71){$c_3$}
      \put(39,24){\vector(-1,3){14}}
      \put(48,30){;}
      \put(15,0){$\mathcal P_3$}
\end{picture}\quad
\begin{picture}(50,90)
      \put(0,15){$a_1$}
      \put(20,15){$b_1$}
      \put(40,15){$c_1$}
      \put(2,24){\vector(0,1){15}}
      \put(22,24){\vector(0,1){15}}
      \put(42,24){\vector(0,1){15}}
      \put(0,44){$a_2$}
      \put(20,44){$b_2$}
      \put(40,44){$c_2$}
      \put(22,52){\vector(0,1){15}}
      \put(42,52){\vector(0,1){15}}
      \put(20,71){$b_3$}
      \put(40,71){$c_3$}
      \put(19,24){\vector(-1,1){15}}
      \put(4,24){\vector(1,3){14}}
      \put(48,30){;}
      \put(15,0){$\mathcal P_4$}
\end{picture}\quad
\begin{picture}(50,110)
      \put(0,15){$a_1$}
      \put(20,15){$b_1$}
      \put(42,15){$c_1$}
      \put(2,24){\vector(0,1){15}}
      \put(22,24){\vector(0,1){15}}
      \put(44,24){\vector(0,1){15}}
      \put(0,44){$a_2$}
      \put(20,44){$b_2$}
      \put(42,44){$c_2$}
      \put(44,52){\vector(0,1){15}}
      \put(40,71){$c_3$}
      \put(44,80){\vector(0,1){15}}
      \put(42,98){$c_4$}
      \put(24,24){\vector(1,4){17.5}}
      \put(20,24){\vector(-1,1){15}}
      \put(50,30){;}
      \put(15,0){$\mathcal P_5$}
\end{picture}\quad
\begin{picture}(50,110)
      \put(0,15){$a_1$}
      \put(20,15){$b_1$}
      \put(40,15){$c_1$}
      \put(2,24){\vector(0,1){15}}
      \put(22,24){\vector(0,1){15}}
      \put(42,24){\vector(0,1){15}}
      \put(0,44){$a_2$}
      \put(20,44){$b_2$}
      \put(40,44){$c_2$}
      \put(42,52){\vector(0,1){15}}
      \put(40,71){$c_3$}
      \put(42,80){\vector(0,1){15}}
      \put(40,98){$c_4$}
      \put(4,24){\vector(1,1){15}}
      \put(39,24){\vector(-1,1){15}}
      \put(48,30){.}
      \put(15,0){$\mathcal P_5^*$}
\end{picture}
\end{center}
Complete list of all quite sincere representations of these posets
are given in the Appendix A.
\end{thm}

\begin{proof}
Let $\pi$ be a quite sincere representation of a poset $\mathcal P$.
The following two possibilities can occur
\begin{itemize}
\item $\pi$ is a sincere representation of $\mathcal P$ with $\pi(i)\neq\pi(0)$;
\item there is $k\in\mathcal P$ such that $\pi(k)=\sum_{i\prec k}\pi(i)$.
\end{itemize}
In the first case  the representation (and obviously the
corresponding poset) is in Kleiner's list of sincere posets and
their representations~\cite{Kleiner2}. In the second the
representation $\pi$ generates an indecomposable representation of
poset the $\mathcal P_1=\mathcal P\setminus {k}$ which we denote by
$\pi_1$. It is clear that $\pi_1$ is a quite sincere representation
of $\mathcal P_1$ and it again satisfies one of the two above
possibilities. Proceeding in this way we will obtain a sincere
representation of some poset $\mathcal P_k$ with condition
$\pi(i)\neq\pi(0)$ which is in Kleiner's list.

Summing up we have the following algorithm for calculation of all
quite sincere posets and their quite sincere representations:
\begin{enumerate}
\item all sincere posets which have a sincere representation
such that $\pi(i)\neq\pi(0)$ are quite sincere and Kleiner's list
gives all. These posets are $(1,1,1)$, $(1,2,2)$, $(1,2,3)$,
$(1,2,4)$, $\mathcal{P}_2$ with representations listed in Appendix A
(for $\mathcal{P}_2$ marked with
*);

\item
Let $\mathcal P$ be a quite sincere poset and and $I\subset \mathcal
P$ such that $\sum_{i\in I}\pi(i)\neq \pi(0)$. Let $\mathcal {\tilde
P}_I$ be the corresponding extended poset defined as in Section 3.
Let $\tilde \pi_I$ be a representation of $\mathcal {\tilde P}_I$
given by $\tilde\pi_I(j)=\pi(j)$ for all $j\in\mathcal P$ and
$\tilde\pi_I(\tilde p)=\sum_{i\in I}\pi(i)$. This is evident that
these representations are quite sincere representations of
corresponding posets $\mathcal {\tilde P}_I$.
\end{enumerate}
In this way by induction one can obtain all quite sincere posets and
all their quite sincere representations. The above procedure
terminates because the dimensions of representations are bounded.
\end{proof}

\begin{rem} Let us remark that the unitarization of a quite sincere
representation is equivalent to unitarization of all indecomposable
representation of poset of finite type.
\end{rem}

\subsection{Proof of Theorem~\ref{mainthmUnit}}

As it was mentioned before in order to prove that all primitive
posets of finite type can be unitarized it is enough to see that all
quite sincere posets and their quite sincere representations are
unitarized. All such representations are listed in Appendix A
(second column in the table). The fact that all quite sincere
representation of primitive poset (from Appendix A) can be
unitarized with some weight is due to \cite{GrushevoyYusenko}.

To prove that quite sincere representation 1)-6) of the poset
$\mathcal P_2$ from Appendix A can be unitarized we show their
stability with the the weight which is equal to the dimension of
representation. For this one can  describe all possible subdimension
for each representation and then check the stability condition.
Description of all possible subdimension is straightforward (this is
done in Appendix D) and this is a routine to check the stability
condition for all these vectors.

Representations of $\mathcal P_1, \mathcal P_3-\mathcal P_5$ and
representations $7)$, $8)$ of $\mathcal P_2$ can be obtained by
adding one (or several) subspace(s) to either a representation of
primitive posets or one of the representation $1)-6)$ of $\mathcal
P_2$. Their unitarization follows from the following observation.

\begin{lem}
    Let $\pi=(V;V_i)$ be a $\chi$-stable system of subspaces. Then for
    each subspace $U\subset V$ there exists a weight $\tilde \chi$
    such that the system of subspaces $\tilde \pi$ generated by $\pi$ and $U$,
    $$\tilde \pi=(V;V_1,\ldots,V_n,U)$$  is $\tilde
    \chi$-stable.
\end{lem}

\begin{proof}
    Let $K\subset V$ be a subspace of $V$ that the difference
    $\lambda_\chi(\pi)-\lambda_\chi(\pi \cap K)$ is minimal.
    Let $R=\lambda_\chi(\pi)-\lambda_\chi(\pi \cap K)$. Since $\pi$ is
    stable, $R>0$ and there exist $\epsilon>0$ that
    $R-\epsilon>0$. Define $\tilde \chi$ in the folowing way
    $$
        \tilde \chi_i=\chi_i, \quad i=1,\dots,n, \quad
        \tilde \chi_{n+1}=R-\epsilon.
    $$
    Our claim is that $\tilde \pi$ is $\tilde\chi$-stable. Indeed,
    let $M\subset V$ be an arbitrary subspace of $V$ then we have
    \begin{equation*}
    \begin{split}
        \frac{1}{\dim M}\sum_{i=1}^{n+1}\tilde \chi_i \dim
        (\tilde \pi(i)\cap M)&=\frac{1}{\dim M}\sum_{i=1}^{n}\chi_i \dim
        (V_i\cap M)+\frac{\tilde\chi_{n+1} \dim(U\cap M)}{\dim M}\\
        &\leq \frac{1}{\dim
        V}\sum_{i=1}^{n}\chi_i\dim V_i-R+ \frac{ (R-\epsilon) \dim(U\cap M)}{\dim
        M}\\
        & < \frac{1}{\dim
        V}\sum_{i=1}^{n}\chi_i\dim V_i < \frac{1}{\dim
        V}\sum_{i=1}^{n+1}\tilde \chi_i\dim \tilde \pi(i).
    \end{split}
    \end{equation*}
    Hence $\tilde \pi$ is $\tilde\chi$-stable.
\end{proof}

\begin{cor}
    If indecomposable system of subspaces $\pi=(V;V_1,\ldots,V_n)$
    is unitarizable then for arbitrary collections of
    subspaces $U_j \subset V$, $j=1,\ldots,m$ the system
    $\pi=(V;V_1,\ldots,V_n,U_1,\ldots,U_m)$ is also unitarizable.
\end{cor}

\begin{rem} Let us note that when a preliminary version of
this article was ready the authors were informed that the same
result was independently obtained in \cite{Ya2010}.
\end{rem}

It remains to prove that the representation of $\mathcal P^*_5$ from
the Appendix A is unitarized. Since it is dual to representation of
$\mathcal P_5$, this follows from the following lemma.

\begin{lem}
    Let $\pi=(V;V_i)$ be unitarizable with the weight $\chi$  and let
    $\pi^\prime=(V;V^\prime_i)$ be indecomposable dual system (each $V^\prime_i$
    is a complement to $V_i$)  assume also that the dimension vector
    of $\pi$ is a real root, i.e. $Q_\mathcal P(d_\pi)=1$. Then  $\pi^\prime$
    is also
    unitarizable with the weight $\chi$.
\end{lem}

\begin{proof}
As $\pi$ is unitarizable with $\chi$, we have $\sum\chi_i
P_{V_i}=\lambda_\pi I$ for appropriated choice of scalar product. It
is not hard to check the dimension vector $d_\pi^\prime$ is also a
real root. Undecomposability of $\pi^\prime$ implies that it is
linearly equivalent to the system $(V;\textrm{Im}(I-P_{V_i}))$
(because there exist only one indecomposable representation with
dimension vector $d_\pi^\prime$). The latter system is obviously
unitarized with the weight $\chi$ due to
$\sum\chi_i(I-P_{V_i})=(\sum\chi_i-\lambda_\pi) I$.
\end{proof}

Now Theorem 2 is completely proved.


\medskip


\section*{Appendix A. Quite sincere representation and weights appropriated for unitarization.}

\renewcommand*\arraystretch{1.5}

In this appendix you can find a complete description of all quite
sincere representations of finite posets and the weights
appropriated for the unitarization. To simplify the notation we will
denote by $V_{i_1,\ldots,i_j}$ vector space spanned by vectors
$e_{i_1},\ldots,e_{i_j}$.

\begin{longtable}{c|p{7.2cm}|p{3.1cm}}
\textbf{Poset} & \textbf{Representation} $\pi=(V;\pi(a_i);\pi(b_i);\pi(c_i))$ & \textbf{Weight $\chi$}\\
\hline
 $(1,1,1)$ & $(\mathbb{C}^2;V_1;V_2;  V_{12})$ & $(1,1,1)$ \\
 \hline
 $(1,2,2)$ & 1) $(\mathbb{C}^3;V_{123};V_1, V_{1,2};V_{3}, V_{2,3})$ & $(1,1,1,1,1)$ \\
 & 2) $(\mathbb{C}^3; V_{12,13};V_{1}, V_{1,2};V_{3}, V_{2,3})$ & $(2,1,1,1,1)$ \\
 \hline
 $(1,2,3)$ & 1) $(\mathbb{C}^4;V_{123,24};V_4, V_{1,4};V_3, V_{2,3}, V_{1,2,3})$ & $(2,1,1,1,1,1)$ \\
 & 2) $(\mathbb{C}^4; V_{124,13};V_4, V_{1,2,4};V_{3}, V_{2,3}, V_{1,2,3})$ & $(2,1,2,1,1,1)$ \\
 & 3) $(\mathbb{C}^4; V_{123,24}; V_{1,4}, V_{1,2,4}; V_{3}, V_{2,3}, V_{1,2,3})$ & $(2,2,1,1,1,1)$ \\
\hline
 $(1,2,4)$ &
   1) $(\mathbb{C}^5;V_{134,235};V_{5}, V_{1,2,5};$\\ & \qquad $V_4, V_{3,4}, V_{2,3,4}, V_{1,2,3,4})$&  $(2,1,2,1,1,1,1)$ \\
 & 2) $(\mathbb{C}^5; V_{123,245};V_{1,5}, V_{1,2,5};$\\ & \qquad $ V_4, V_{3,4}, V_{2,3,4}, V_{1,2,3,4})$ &  $(2,2,1,1,1,1,1)$ \\
 & 3) $(\mathbb{C}^5; V_{124,235};V_{1,5}, V_{1,2,3,5};$\\ & \qquad $V_4, V_{3,4}, V_{2,3,4}, V_{1,2,3,4})$ & $(2,2,2,1,1,1,1)$  \\
 & 4) $(\mathbb{C}^5; V_{124,23,15}; V_5, V_{1,2,5};$\\ & \qquad $V_4, V_{3,4}, V_{2,3,4}, V_{1,2,3,4})$ &  $(3,1,2,1,1,1,1)$ \\
 & 5) $(\mathbb{C}^5; V_{13,234,45};V_{1,5}, V_{1,2,5};$\\ & \qquad $ V_4, V_{3,4}, V_{2,3,4}, V_{1,2,3,4})$ & $(3,2,1,1,1,1,1)$ \\
 & 6) $(\mathbb{C}^5; V_{12,234,45}; V_{1,5}, V_{1,2,3,5};$\\ & \qquad $V_4, V_{3,4}, V_{2,3,4}, V_{1,2,3,4})$ & $(3,2,2,1,1,1,1)$  \\
 & 7) $(\mathbb{C}^6; V_{123,245,16}; V_{5,6}, V_{1,2,5,6};$\\ & \qquad $ V_4, V_{3,4}, V_{2,3,4}, V_{1,2,3,4})$ & $(3,2,2,1,1,1,1)$ \\
 & 8) $(\mathbb{C}^6; V_{125,234,46}; V_{1,6}, V_{1,2,3,6};$\\ & \qquad $ V_5, V_{4,5}, V_{3,4,5}, V_{1,2,3,4,5})$ & $(3,2,2,1,1,1,2)$ \\
 & 9) $(\mathbb{C}^6; V_{125,134,46}; V_{1,6}, V_{1,2,3,6};$\\ & \qquad $ V_{5}, V_{4,5}, V_{2,3,4,5}, V_{1,2,3,4,5})$ & $(3,2,2,1,1,2,1)$ \\
 & 10) $(\mathbb{C}^6; V_{125,234,46}; V_{1,6}, V_{1,2,3,6};$\\ & \qquad $ V_5, V_{3,4,5}, V_{2,3,4,5}, V_{1,2,3,4,5})$ & $(3,2,2,1,2,1,1)$\\
 & 11) $(\mathbb{C}^6; V_{135,124,46};V_{1,6},V_{1,2,3,6};$\\ & \qquad $V_{4,5}, V_{3,4,5}, V_{2,3,4,5},  V_{1,2,3,4,5})$ & $(3,2,2,2,1,1,1)$  \\
\hline
$\mathcal P_1$ &  $(\mathbb{C}^3;V_{123}, V_{23,1};V_1, V_{1,2};V_{3},V_{2,3})$ & $(1,0.1,1,1,1,1)$ \\
 \hline
$\mathcal P_2$ & 1) $(\mathbb{C}^4;V_{14}, V_{1,2,4}; V_4, V_{4,123};V_3, V_{2,3}, V_{1,2,3})^*$ & $(1,1,1,1,1,1,1)$\\
 & 2) $(\mathbb{C}^4;V_{14}, V_{1,2,4}; V_4, V_{4,12,23}; V_3, V_{2,3}, V_{1,2,3})^*$ & $(1,1,1,2,1,1,1)$ \\
 & 3) $(\mathbb{C}^5; V_{1,25}, V_{1,2,3,5}; V_5, V_{123,24,5};$\\ & \qquad $ V_{3,4}, V_{2,3,4},V_{1,2,3,4})^*$ & $(2,1,1,2,2,1,1)$\\
 & 4) $(\mathbb{C}^5; V_{1,25}, V_{1,2,3,5}; V_5, V_{13,234,5};$\\ & \qquad $ V_4, V_{2,3,4}, V_{1,2,3,4})^*$ & $(2,1,1,2,1,2,1)$ \\
 & 5) $(\mathbb{C}^5; V_{1,25}, V_{1,2,3,5}; V_5, V_{123,24,5};$\\ & \qquad $ V_4, V_{3,4},  V_{1,2,3,4})^*$ & $(2,1,1,2,1,1,2)$ \\
 & 6) $(\mathbb{C}^5; V_{15,4}, V_{1,2,4,5}; V_5, V_{123,24,5};$\\ & \qquad $ V_3, V_{2,3},  V_{1,2,3})^*$ & $(2,1,1,2,1,1,1)$\\
 & 7) $(\mathbb{C}^4; V_{123,24}, V_{13,2,4}; V_4, V_{1,4};$\\ &\qquad $V_3, V_{2,3}, V_{1,2,3})$ & $(2,0.1,1,1,1,1,1)$ \\
 & 8) $(\mathbb{C}^4; V_{124,13}, V_{12,13,4}; V_4, V_{1,2,4};$\\ &\qquad $V_3, V_{2,3}, V_{1,2,3})$ & $(2,0.1,1,2,1,1,1)$ \\
\hline
$\mathcal{P}_3$ & $(\mathbb{C}^4;V_{123,24};V_4, V_{1,4}, V_{1,3,4}; V_3, V_{2,3}, V_{1,2,3})$ & $(2,1,1,0.1,1,1,1)$ \\
\hline
$\mathcal P_4$ & $(\mathbb{C}^4;V_{14}, V_{1,2,4}; V_4, V_{4,123}, V_{1,23,4};$\\ &\qquad $ V_3, V_{2,3}, V_{1,2,3})$  & $(1,1,1,1,0.1,1,1,1)$\\
\hline
$\mathcal P_5$ & $(\mathbb{C}^5; V_{15,4}, V_{1,2,4,5}; V_5, V_{123,24,5};$\\ & \qquad $ V_3, V_{2,3},  V_{1,2,3}, V_{1,2,3,5})$  & $(2,1,1,2,1,1,1,0.1)$ \\
\hline
$\mathcal P_5^*$ & $(\mathbb{C}^5;V_{5}, V_{1,2,5}; V_{134,235}, V_{13,23,4,5};$\\ & \qquad $ V_4, V_{3,4}, V_{2,3,4}, V_{1,2,3,4})$ & $(1,2,2,0.1,1,1,1,1)$  \\
\hline
\end{longtable}

\begin{rem}
    An interesting phenomena that each quite sincere representation can be unitarized with the weight that equal to
    the dimension vector (for primitive case) or with the weight that an arbitrary closed to the dimension vector (for
non-primitive poset) and this weight in some sense is the most
stable weight.
\end{rem}

\section*{Appendix B. The sets $\triangle_{\pi}^{\mathcal P}$ and several examples.}

It is routine to describe the set $\triangle_{\pi}^{\mathcal P}$ for
an arbitrary linear representation $\pi$ of $\mathcal P$. Instead we
will give an algorithm of its description.

\begin{prop} $\triangle_{\pi}^{\mathcal P}$ is convex.
\end{prop}
\begin{proof}
One can see that if $\chi \in \triangle_{\pi}^{\mathcal P}$ then
$(1-t)\chi \in \triangle_{\pi}^{\mathcal P}$ for each $t\in [0,1]$,
because if $\pi$ is stable with $\chi$ then $\pi$ is also stable
with $(1-t)\chi$. Hence $\triangle_{\pi}^{\mathcal P}$ is convex and
connected.
\end{proof}

Stability conditions for the system of subspaces
$\pi=(V;V_1,\ldots,V_n)$ define some $(m\times n)$--matrix $A_\pi\in
M_{m,n}(\mathbb R)$. Namely this matrix is defined in the following
way. For any vector of the form $d=(d_0;d_1,\ldots,d_n)\in \mathbb
Z^{n+1}$ with $d_0>0$ by $n(d)$ we denote normalized vector
$n(d)=\left(\frac{d_1}{d_0},\ldots,\frac{d_n}{d_0}\right)$. Let
$\dim(\pi)=(\dim \pi_0;\dim \pi_i)$ be dimension vector of $\pi$ and
$Sub(\pi)=\{d_{\pi,i}\ |\ i\in \{1,\ldots,m\} \}$ be the set of
maximal subdimension vectors for $\pi$. Then the matrix $A_\pi$ has
the following form
$$
       A_\pi= \left(
          \begin{array}{c}
            n(d_{\pi,1})-n(\dim(\pi))  \\
            \vdots \\
            n(d_{\pi,m})-n(\dim(\pi)) \\
          \end{array}
        \right) \oplus -I_n,
$$
here by $\oplus$ be understand  The set $\triangle_{\pi}^{\mathcal
P}$ thus can be defined as the set of those
$\chi=(\alpha_1,\ldots,\alpha_n)$ that $A_\pi\chi<0$. Using the
standard methods concerning to systems of linear inequalities (see
for example \cite{PadbergManfred}) these sets can be described in
terms of extreme points and extremal rays.

\begin{defn}
Let $P \subset \mathbb R^n$ be some subset. A point $x \in P$ is
called an extreme point of P if for all $x_1,x_2 \in P$ and every
$0<\mu<1$ such that $x=\mu x_1 +(1-\mu) x_2$, we have $x=x_1=x_2$.
\end{defn}

The set of extreme points can be determined by the set of those
points $x_0 \in \mathbb R^n$ that $A_\pi^\prime x_0=0$ for some
$(n\times n)$ full rank submatrix $A_\pi^\prime$ (see
\cite{PadbergManfred}), hence this set contains the only element
$x_0=0$.

\begin{defn}
    \begin{enumerate}
        \item A set $C \subset \mathbb R^n$ is a cone if for every
        pair of points $x_1,x_2 \in C$ we have
        $\lambda_1x_1+\lambda_2x_2 \in C$ for all
        $\lambda_1,\lambda_2\geq0$.
        \item A half-line $y=\{\lambda x \ | \ \lambda \geq 0, \ x
        \in \mathbb R^n\}$ is an extremal ray of $C$ if $y \in C$
        and $-y \notin C$ and if for all $y_1,y_2 \in C$ and $0<\mu<1$
        with $y=(1-\mu)y_1+\mu y_2$ we have $y=y_1=y_2$.
    \end{enumerate}
\end{defn}

Obviously $\triangle_{\pi}^{\mathcal P}$ is a cone. The following
proposition describes how to determine all extermal rays of
$\triangle_{\pi}^{\mathcal P}$.

\begin{prop} (see \cite{PadbergManfred})
    $x \in \overline{\triangle_{\pi}^{\mathcal P}}$ is an extremal ray if and
    only if  there exist $rank(A_\pi)-1$ linear independent row vectors
    $a_1,\ldots,a_{rank(A_\pi)-1}$ of $A_\pi$ such that
    $$
        \left(
          \begin{array}{c}
            a_1 \\
            \vdots \\
            a_{rank(A_\pi)-1} \\
          \end{array}
        \right) \cdot x =\vec{0}.
    $$
\end{prop}

\begin{prop} (see \cite{PadbergManfred})
The sets $\triangle_{\pi}^{\mathcal P}$ are described in the
following way
$$
    \triangle_{\pi}^{\mathcal P}=cone(A),
$$
where $A$ is the set of extremal rays, $cone(X)$ is an open cone of
$X$ defined by
$$cone(x_1,\ldots,x_n)=\left\{\sum_{i=1}^n \mu_i x_i\ |\ \mu_i>0\right\}.$$
\end{prop}

We consider two examples:

\emph{1)}. Let $\mathcal P=(1,1,1)$, and $\pi=(\mathbb
C^2;V_1,V_2,V_{12})$. Matrix $A_{\pi}$ is given by
$$
       A_\pi= \left(
          \begin{array}{ccc}
            \frac{1}{2} & -\frac{1}{2} & -\frac{1}{2} \\
            -\frac{1}{2} & \frac{1}{2} & -\frac{1}{2} \\
            -\frac{1}{2} & -\frac{1}{2} & \frac{1}{2} \\
          \end{array}
        \right)\oplus -I_3.
$$
The set $\triangle^{(1,1,1)}_{\pi}$ has the only extreme point
$(0,0,0)$ and three extremal rays $(1,1,0),\ (1,0,1),\ (0,1,1)$.
Hence the whole set is given by
$$\triangle^{(1,1,1)}_{(\mathbb
C^2;V_1,V_2,V_{12})}=\{(\alpha_1+\alpha_2,\alpha_1+\alpha_3,\alpha_2+\alpha_3)
\ \mid \ \alpha_i \in \mathbb R_+\};$$

\emph{2)}.  Let us take the poset $\mathcal
P=\mathcal(N,2)=\{a_1,a_2,b_1,b_2,c_1,c_2\}$, $a_1\prec a_2$,
$b_1\prec b_2$, $b_1\prec a_2$, $c_1\prec c_2$. It has the only
quite sincere representation $\pi$:

\begin{center}
\begin{picture}(150,50)
      \put(4,5){$\langle e_{123}\rangle$}
      \put(55,5){$\langle e_1 \rangle$}
      \put(98,5){$\langle e_3 \rangle$}
      \put(15,16){\vector(0,1){20}}
      \put(63,16){\vector(0,1){20}}
      \put(106,16){\vector(0,1){20}}
      \put(0,40){$\langle e_1,e_{123}\rangle$}
      \put(48,40){$\langle e_1, e_2\rangle$}
      \put(91,40){$\langle e_2, e_3 \rangle$}
      \put(59,16){\vector(-2,1){35}}
\end{picture}
\end{center}

Let $\chi=(\alpha_1,\alpha_2,\beta_1,\beta_2,\gamma_1,\gamma_2)$.
The normalized dimension vector of $\pi$ is equal to
$\left(\frac{1}{3},\frac{2}{3},\frac{1}{3},\frac{2}{3},\frac{1}{3},\frac{2}{3}\right)$.
The set of maximal subdimension vectors is
\begin{align*}
Sub(\pi)=\{&(1;0,0;0,0;1,1),\ (1;0,0;0,1;0,1),\ (1;0,1;0,0;0,1),\
(1;0,1;1,1;0,0),\\ &(1;1,1;0,0;0,0),\ (2;0,1;0,1;1,2),\
(2;0,1;1,1;1,1),\ (2;0,1;1,2;0,1),\\ &(2;1,1;0,1;1,1),\
(2;1,2;1,1;0,1)\}.
\end{align*}
The corresponding matrix $A_\pi$ has the following form
$$
       A_\pi= \left(
          \begin{array}{cccccc}
              -\frac{1}{3} & -\frac{2}{3}
              &-\frac{1}{3}&-\frac{2}{3}&\frac{2}{3}&\frac{1}{3}\\
              -\frac{1}{3} & -\frac{2}{3}
              &-\frac{1}{3}& \frac{1}{3}&-\frac{1}{3}&\frac{1}{3}\\
              -\frac{1}{3} & \frac{1}{3}
              &-\frac{1}{3}&-\frac{2}{3}&-\frac{1}{3}&\frac{1}{3}\\
              -\frac{1}{3} & \frac{1}{3}
              &\frac{2}{3}& \frac{1}{3}&-\frac{1}{3}&-\frac{2}{3}\\
              \frac{2}{3} & \frac{1}{3}
              &-\frac{1}{3}&-\frac{2}{3}&-\frac{1}{3}&-\frac{2}{3}\\
              -\frac{1}{3} & -\frac{1}{6}
              &-\frac{1}{3} & -\frac{1}{6}&\frac{1}{6}&\frac{1}{3}\\
              -\frac{1}{3} & -\frac{1}{6}
              &\frac{1}{6} & -\frac{1}{6}&\frac{1}{6}&-\frac{1}{6}\\
              -\frac{1}{3} & -\frac{1}{6}
              &\frac{1}{6} & \frac{1}{3}&-\frac{1}{3}&-\frac{1}{6}\\
              \frac{1}{6} & -\frac{1}{6}
              &-\frac{1}{3} & -\frac{1}{6}&\frac{1}{6}&-\frac{1}{6}\\
              \frac{1}{6} & \frac{1}{3}
              &\frac{1}{6} & -\frac{1}{6}&-\frac{1}{3}&-\frac{1}{6}\\
          \end{array}
          \right )\oplus -I_6.
$$
The set $\triangle^{(N,2)}_{\pi}$ has nine extremal rays
\begin{gather*}
(0,0,0,1,1,1),\ (0,0,1,0,0,1),\ (0,1,0,0,1,0),\ (0,1,0,1,0,1),\\
(0,2,1,1,3,1),\ (1,0,0,0,0,1),\ (2,0,0,1,1,2),\ (1,0,0,1,1,1),\\
(1,0,1,0,1,0).
\end{gather*} Hence the whole set is given by
\begin{align*}
\triangle^{(N,2)}_{\pi}=\{(&\alpha_6+2\alpha_7+\alpha_8+\alpha_9,\alpha_3+\alpha_4+2\alpha_5,\alpha_2+\alpha_5+\alpha_9,\\
                           &\alpha_1+\alpha_4+\alpha_5+\alpha_7+\alpha_8,\alpha_1+\alpha_3+3\alpha_5+\alpha_7+\alpha_8+\alpha_9,\\
                           &\alpha_2+\alpha_4+\alpha_5+\alpha_6+2\alpha_7+\alpha_8) \ \mid \ \alpha_i \in \mathbb
                           R_+\}.
\end{align*}


\begin{rem}
 One can see that $(1,1,1,1,1,1)  \notin \triangle_{\pi}^{(N,2)}$ (this vector is semistable but is not stable) which
 means that given linear representation of the poset $(N,2)$
 can not be obtained as the spectral filtration of three partial
 reflections $A_i=A_i^*=A_i^3$ sum of which is zero $A_1+A_2+A_3=0$.
\end{rem}

\section*{Appendix C. Missing details in the proof of Theorem 1.}

\begin{lem}
The following lists contains all possible subdimension for extended
representations used in the proof of Theorem 1.

\begin{longtable}{c|p{4cm}|p{6cm}}
\bf{Poset} & \bf{Subdimensions}

$d=(\dim U;\dim(V_i\cup U))$ & \bf{Subspace $U$}\\
\hline $(1,1,1,1)$ & $(1;1;0;0;0)$ & $\langle e_1 \rangle$ \\
                   & $(1;0;1;0;0)$ & $\langle e_2 \rangle$ \\
                   & $(1;0;0;1;0)$ & $\langle e_1+e_2 \rangle$ \\
                   & $(1;0;0;0;1)$ & $\langle e_1+\lambda e_2 \rangle$ \\
\hline $(2,2,2)$ & $(1;0,0;0,0;1,1)$   &  $\langle e_3\rangle$ \\

& $(1;0,0;0,1;0,1)$   &  $\langle e_2\rangle$ \\

& $(1;0,1;0,0;0,1)$   &  $\langle e_2+(\lambda-1) e_3\rangle$ \\

& $(1;0,0;1,1;0,0)$   &  $\langle e_1\rangle$ \\

& $(1;0,1;0,1;0,0)$   &  $\langle (\lambda-1)e_1+\lambda e_2\rangle$ \\

& $(1;1,1;0,0;0,0)$   &  $\langle e_1+e_2+e_3\rangle$ \\

& $(2;0,1;0,1;1,2)$   &  $\langle e_3,e_2\rangle$ \\

& $(2;0,1;1,1;1,1)$   &  $\langle e_3,e_1\rangle$ \\

& $(2;0,1;1,2;0,1)$   &  $\langle e_2,e_1\rangle$ \\

& $(2;1,1;0,1;1,1)$   &  $\langle e_3,e_1+e_2+e_3\rangle$ \\

& $(2;1,1;1,1;0,1)$   &  $\langle e_1,e_1+e_2+e_3\rangle$ \\

& $(2;1,2;0,1;0,1)$   &  $\langle e_1+\lambda e_3,e_1+e_2+e_3\rangle$ \\

\hline $(1;3;3)$

& $(1;0;0,0,0;1,1,1)$   &  $\langle e_3\rangle$ \\

& $(1;0;0,0,1;0,1,1)$   &  $\langle e_2+\lambda e_3\rangle$ \\

& $(1;0;0,1,1;0,0,1)$   &  $\langle e_1\rangle$ \\

& $(1;0;1,1,1;0,0,0)$   &  $\langle e_4\rangle$ \\

& $(1;1;0,0,0;0,0,1)$   &  $\langle e_1+e_2+e_3\rangle$ \\

& $(1;1;0,0,1;0,0,0)$   &  $\langle \lambda e_1+e_2+ \lambda e_3+(1-\lambda) e_4\rangle$ \\

& $(2;0;0,0,1;1,2,2)$   &  $\langle e_3,e_2\rangle$ \\

& $(2;0;0,1,1;1,1,2)$   &  $\langle e_3,e_1\rangle$ \\

& $(2;0;0,1,2;0,1,2)$   &  $\langle e_2+\lambda e_3,e_1\rangle$ \\

& $(2;0;1,1,1;1,1,1)$   &  $\langle e_4,e_3\rangle$ \\

& $(2;0;1,1,2;0,1,1)$   &  $\langle e_4,e_2+\lambda e_3\rangle$ \\

& $(2;0;1,2,2;0,0,1)$   &  $\langle e_4,e_1\rangle$ \\

& $(2;1;0,0,1;1,1,2)$   &  $\langle e_3,e_1+e_2+e_3\rangle$ \\

& $(2;1;0,1,1;0,1,2)$   &  $\langle e_1,e_1+e_2+e_3\rangle$ \\

& $(2;1;1,1,1;0,1,1)$   &  $\langle e_4,e_2\rangle$ \\

& $(2;1;1,1,2;0,0,1)$   &  $\langle \lambda e_1+e_2+ \lambda e_3+(1-\lambda) e_4, e_4  \rangle$ \\

& $(2;1;0,1,2;0,1,1)$   &  $\langle \lambda e_1+e_2+ \lambda e_3+(1-\lambda) e_4, e_2+\lambda e_3 \rangle$ \\

& $(2;1;0,1,1;1,1,1)$   &  $\langle e_3,e_1+e_2-e_4 \rangle$ \\

& $(2;2;0,0,1;0,0,1)$   &  $\langle e_2+e_4,e_1+e_2+e_3\rangle$ \\

& $(3;1;0,1,2;1,2,3)$   &  $\langle e_3,e_2,e_1\rangle$ \\

& $(3;1;1,1,2;1,2,2)$   &  $\langle e_4,e_3,e_2\rangle$ \\

& $(3;1;1,2,2;1,1,2)$   &  $\langle e_4,e_3,e_1\rangle$ \\

& $(3;1;1,2,3;0,1,2)$   &  $\langle e_4,e_2+\lambda e_3,e_1\rangle$ \\

& $(3;2;0,1,2;1,1,2)$   &  $\langle e_3,e_2+e_4,e_1+e_2+e_3\rangle$ \\

& $(3;2;1,1,2;0,1,2)$   &  $\langle e_4,e_2,e_1+e_2+e_3\rangle$ \\

\hline $(1,2,5)$

& $(1;0;0,0;1,1,1,1,1)$   &  $\langle e_4\rangle$ \\

& $(1;0;0,1;0,0,1,1,1)$   &  $\langle e_2\rangle$ \\

& $(1;0;1,1;0,0,0,0,1)$   &  $\langle e_5+2e_6\rangle$ \\

& $(1;1;0,0;0,0,0,1,1)$   &  $\langle e_1+e_2+e_3\rangle$ \\

& $(1;1;0,1;0,0,0,0,0)$   &  $\langle e_1+e_6\rangle$ \\

& $(2;0;0,0;1,2,2,2,2)$   &  $\langle e_4,e_3\rangle$ \\

& $(2;0;0,1;1,1,2,2,2)$   &  $\langle e_4,e_2\rangle$ \\

& $(2;0;0,2;0,0,1,2,2)$   &  $\langle e_2,e_1\rangle$ \\

& $(2;0;1,1;1,1,1,1,2)$   &  $\langle e_5+\lambda e_6,e_4\rangle$ \\

& $(2;0;1,2;0,0,1,1,2)$   &  $\langle e_5+\lambda e_6,e_2\rangle$ \\

& $(2;0;2,2;0,0,0,0,1)$   &  $\langle e_6,e_5\rangle$ \\

& $(2;1;0,0;1,1,1,2,2)$   &  $\langle e_4,e_1+e_2+e_3\rangle$ \\

& $(2;1;0,1;0,1,1,2,2)$   &  $\langle e_3,e_1+e_2+e_3\rangle$ \\

& $(2;1;0,1;1,1,1,1,1)$   &  $\langle e_4,e_2+e_4+e_5\rangle$ \\

& $(2;1;0,2;0,0,1,1,1)$   &  $\langle e_2,e_1+e_6\rangle$ \\

& $(2;1;1,1;0,0,0,1,2)$   &  $\langle e_5+\lambda e_6,e_1+e_2+e_3\rangle$ \\

& $(2;1;1,1;0,0,1,1,1)$   &  $\langle e_5,e_2+e_4+e_5\rangle$ \\

& $(2;1;1,2;0,0,0,1,1)$   &  $\langle e_6,e_1\rangle$ \\

& $(2;2;0,1;0,0,0,1,1)$   &  $\langle e_1+e_6,e_1+e_2+e_3\rangle$ \\

& $(3;0;0,1;1,2,3,3,3)$   &  $\langle e_4,e_3,e_2\rangle$ \\

& $(3;0;0,2;1,1,2,3,3)$   &  $\langle e_4,e_2,e_1\rangle$ \\

& $(3;0;1,1;1,2,2,2,3)$   &  $\langle e_5+\lambda e_6,e_4,e_3\rangle$ \\

& $(3;0;1,2;1,1,2,2,3)$   &  $\langle e_5+\lambda e_6,e_4,e_2\rangle$ \\

& $(3;0;1,3;0,0,1,2,3)$   &  $\langle e_5+\lambda e_6,e_2,e_1\rangle$ \\

& $(3;0;2,2;1,1,1,1,2)$   &  $\langle e_6,e_5,e_4\rangle$ \\

& $(3;0;2,3;0,0,1,1,2)$   &  $\langle e_6,e_5,e_2\rangle$ \\

& $(3;1;0,1;1,2,2,3,3)$   &  $\langle e_4,e_3,e_1+e_2+e_3\rangle$ \\

& $(3;1;0,2;0,1,2,3,3)$   &  $\langle e_3,e_2,e_1\rangle$ \\

& $(3;1;1,1;1,1,1,2,3)$   &  $\langle e_5+\lambda e_6,e_4,e_1+e_2+e_3\rangle$ \\

& $(3;1;1,2;0,1,1,2,3)$   &  $\langle e_5+\lambda e_6,e_3,e_1+e_2+e_3\rangle$ \\

& $(3;1;1,2;1,1,2,2,2)$   &  $\langle e_5,e_4,e_2\rangle$ \\

& $(3;1;1,3;0,0,1,2,2)$   &  $\langle e_6,e_2,e_1\rangle$ \\

& $(3;1;2,2;0,0,1,1,2)$   &  $\langle e_6,e_5,e_2+e_4+e_5\rangle$ \\

& $(3;1;2,3;0,0,0,1,2)$   &  $\langle e_6,e_5,e_1\rangle$ \\

& $(3;2;0,1;1,1,1,2,2)$   &  $\langle e_4,e_2+e_4+e_5,e_1+e_2+e_3\rangle$ \\

& $(3;2;0,2;0,1,1,2,2)$   &  $\langle e_3,e_1+e_6,e_1+e_2+e_3\rangle$ \\

& $(3;2;0,2;1,1,1,1,2)$   &  $\langle e_4,e_2+e_4+e_5,e_1+e_6\rangle$ \\

& $(3;2;1,2;0,0,1,2,2)$   &  $\langle e_6,e_1,e_1+e_2+e_3\rangle$ \\

& $(3;3;0,1;0,0,0,1,2)$   &  $\langle e_2+e_4+e_5,e_1+e_6,e_1+e_2+e_3\rangle$ \\

& $(4;1;0,2;1,2,3,4,4)$   &  $\langle e_4,e_3,e_2,e_1\rangle$ \\

& $(4;1;1,2;1,2,3,3,4)$   &  $\langle e_5+\lambda e_6,e_4,e_3,e_2\rangle$ \\

& $(4;1;1,3;1,1,2,3,4)$   &  $\langle e_5+\lambda e_6,e_4,e_2,e_1\rangle$ \\

& $(4;1;2,2;1,2,2,2,3)$   &  $\langle e_6,e_5,e_4,e_3\rangle$ \\

& $(4;1;2,3;1,1,2,2,3)$   &  $\langle e_6,e_5,e_4,e_2\rangle$ \\

& $(4;1;2,4;0,0,1,2,3)$   &  $\langle e_6,e_5,e_2,e_1\rangle$ \\

& $(4;2;0,2;1,2,2,3,3)$   &  $\langle e_4,e_3,e_2+e_4+e_5,e_1+e_2+e_3\rangle$ \\

& $(4;2;1,2;1,1,2,3,3)$   &  $\langle e_6,e_4,e_1,e_1+e_2+e_3\rangle$ \\

& $(4;2;1,3;0,1,2,3,3)$   &  $\langle e_6,e_3,e_2,e_1\rangle$ \\

& $(4;2;1,3;1,1,2,2,3)$   &  $\langle e_5,e_4,e_2,e_1+e_6\rangle$ \\

& $(4;2;2,3;0,0,1,2,3)$   &  $\langle e_6,e_5,e_2+e_4+e_5,e_1\rangle$ \\

& $(4;3;0,2;1,1,1,2,3)$   &  $\langle e_4,e_2+e_4+e_5,e_1+e_6,e_1+e_2+e_3\rangle$ \\

& $(4;3;1,2;0,0,1,2,3)$   &  $\langle e_6,e_2+e_4+e_5,e_1,e_1+e_2+e_3\rangle$ \\

& $(5;2;1,3;1,2,3,4,5)$   &  $\langle e_5+\lambda e_6,e_4,e_3,e_2,e_1\rangle$ \\

& $(5;2;2,3;1,2,3,3,4)$   &  $\langle e_6,e_5,e_4,e_3,e_2\rangle$ \\

& $(5;2;2,4;1,1,2,3,4)$   &  $\langle e_6,e_5,e_4,e_2,e_1\rangle$ \\

& $(5;3;1,3;1,2,2,3,4)$   &  $\langle e_4,e_3,e_2+e_4+e_5,e_1+e_6,e_1+e_2+e_3\rangle$ \\

& $(5;3;2,3;0,1,2,3,4)$   &  $\langle e_6,e_5,e_2+e_4+e_5,e_1,e_1+e_2+e_3\rangle$ \\

\hline $(N,4)$
& $(1;0,0;0,0;1,1,1,1)$   &  $\langle e_4\rangle$ \\

& $(1;0,0;0,1;0,0,1,1)$   &  $\langle e_2\rangle$ \\

& $(1;0,1;0,0;0,1,1,1)$   &  $\langle e_3+\lambda e_4\rangle$ \\

& $(1;0,1;1,1;0,0,0,0)$   &  $\langle e_5\rangle$ \\

& $(1;1,1;0,0;0,0,0,1)$   &  $\langle e_1+e_3+e_4\rangle$ \\

& $(2;0,1;0,0;1,2,2,2)$   &  $\langle e_4,e_3\rangle$ \\

& $(2;0,1;0,1;1,1,2,2)$   &  $\langle e_4,e_2\rangle$ \\

& $(2;0,1;0,2;0,0,1,2)$   &  $\langle e_2,e_1\rangle$ \\

& $(2;0,1;1,1;1,1,1,1)$   &  $\langle e_5,e_4\rangle$ \\

& $(2;0,1;1,2;0,0,1,1)$   &  $\langle e_5,e_2\rangle$ \\

& $(2;0,2;1,1;0,1,1,1)$   &  $\langle e_5,e_3+\lambda e_4\rangle$ \\

& $(2;1,1;0,0;1,1,1,2)$   &  $\langle e_4,e_1+e_3+e_4\rangle$ \\

& $(2;1,1;0,1;0,1,1,2)$   &  $\langle e_1,e_1+e_3+e_4\rangle$ \\

& $(2;1,2;0,0;0,1,1,2)$   &  $\langle e_3+\lambda e_4,e_1+e_3+e_4\rangle$ \\

& $(2;1,2;1,1;0,0,1,1)$   &  $\langle e_5,e_2+e_3+e_5\rangle$ \\

& $(2;2,2;0,0;0,0,0,1)$   &  $\langle e_2+e_3+e_5,e_1+e_3+e_4\rangle$ \\

& $(3;0,2;0,1;1,2,3,3)$   &  $\langle e_4,e_3,e_2\rangle$ \\

& $(3;0,2;0,2;1,1,2,3)$   &  $\langle e_4,e_2,e_1\rangle$ \\

& $(3;0,2;1,1;1,2,2,2)$   &  $\langle e_5,e_4,e_3\rangle$ \\

& $(3;0,2;1,2;1,1,2,2)$   &  $\langle e_5,e_4,e_2\rangle$ \\

& $(3;0,2;1,3;0,0,1,2)$   &  $\langle e_5,e_2,e_1\rangle$ \\

& $(3;1,2;0,1;1,2,2,3)$   &  $\langle e_4,e_3,e_1\rangle$ \\

& $(3;1,2;0,2;0,1,2,3)$   &  $\langle e_2,e_1,e_1+e_3+e_4\rangle$ \\

& $(3;1,2;1,1;1,1,2,2)$   &  $\langle e_5,e_4,e_2+e_3+e_5\rangle$ \\

& $(3;1,2;1,2;0,1,2,2)$   &  $\langle e_5,e_3,e_2\rangle$ \\

& $(3;1,3;1,1;0,1,2,2)$   &  $\langle e_5,e_3+\lambda e_4,e_2+e_3+e_5\rangle$ \\

& $(3;2,2;0,1;1,1,1,2)$   &  $\langle e_4,e_2+e_3+e_5,e_1+e_3+e_4\rangle$ \\

& $(3;2,3;0,1;0,1,1,2)$   &  $\langle e_3+\lambda e_4,e_2+e_3+e_5,e_1+e_3+e_4\rangle$ \\

& $(3;2,3;1,1;0,0,1,2)$   &  $\langle e_5,e_2+e_3+e_5,e_1+e_3+e_4\rangle$ \\

& $(4;1,3;0,2;1,2,3,4)$   &  $\langle e_4,e_3,e_2,e_1\rangle$ \\

& $(4;1,3;1,2;1,2,3,3)$   &  $\langle e_5,e_4,e_3,e_2\rangle$ \\

& $(4;1,3;1,3;1,1,2,3)$   &  $\langle e_5,e_4,e_2,e_1\rangle$ \\

& $(4;2,3;0,2;1,2,2,3)$   &  $\langle e_4,e_3,e_2+e_3+e_5,e_1\rangle$ \\

& $(4;2,3;1,2;1,1,2,3)$   &  $\langle e_5,e_4,e_2+e_3+e_5,e_1+e_3+e_4\rangle$ \\

& $(4;2,4;1,2;0,1,2,3)$   &  $\langle e_5,e_3+\lambda e_4,e_2+e_3+e_5,e_1+e_3+e_4\rangle$ \\

\end{longtable}

\end{lem}

It is routine to check that for the corresponding representations
stability  conditions (where weight is taken to be dimension) holds
for all maximal subdimension listed above, hence representation are
unitarizable.

\section*{Appendix D. Missing details in the proof of Theorem 2.}

\begin{lem}
The following lists contain all possible subdimension for
representations 1)-6) of the poset $\mathcal P_2$.

\begin{longtable}{c|p{4cm}|p{6cm}}
\bf{No.} & \bf{Subdimensions}

$d=(\dim U;\dim(V_i\cup U))$ & \bf{Subspace $U$}\\

\hline $1)$ & $(1;0,0;0,0;1,1,1)$   &  $\langle e_3\rangle$ \\

& $(1;0,0;0,1;0,0,1)$   &  $\langle e_1+e_2+e_3\rangle$ \\

& $(1;0,1;0,0;0,1,1)$   &  $\langle e_2\rangle$ \\

& $(1;0,1;1,1;0,0,0)$   &  $\langle e_4\rangle$ \\

& $(1;1,1;0,0;0,0,0)$   &  $\langle e_1+e_4\rangle$ \\

& $(2;0,1;0,0;1,2,2)$   &  $\langle e_3,e_2\rangle$ \\

& $(2;0,1;0,1;1,1,2)$   &  $\langle e_3,e_1+e_2+e_3\rangle$ \\

& $(2;0,1;1,1;1,1,1)$   &  $\langle e_4,e_3\rangle$ \\

& $(2;0,1;1,2;0,0,1)$   &  $\langle e_4,e_1+e_2+e_3\rangle$ \\

& $(2;0,2;0,0;0,1,2)$   &  $\langle e_2,e_1\rangle$ \\

& $(2;0,2;1,1;0,1,1)$   &  $\langle e_4,e_2\rangle$ \\

& $(2;1,1;0,0;1,1,1)$   &  $\langle e_3,e_1+e_4\rangle$ \\

& $(2;1,2;0,0;0,1,1)$   &  $\langle e_2,e_1+e_4\rangle$ \\

& $(2;1,2;1,1;0,0,1)$   &  $\langle e_4,e_1\rangle$ \\

& $(3;0,2;0,1;1,2,3)$   &  $\langle e_3,e_2,e_1\rangle$ \\

& $(3;0,2;1,1;1,2,2)$   &  $\langle e_4,e_3,e_2\rangle$ \\

& $(3;0,2;1,2;1,1,2)$   &  $\langle e_4,e_3,e_1+e_2+e_3\rangle$ \\

& $(3;1,2;0,1;1,2,2)$   &  $\langle e_3,e_2,e_1+e_4\rangle$ \\

& $(3;1,2;1,1;1,1,2)$   &  $\langle e_4,e_3,e_1\rangle$ \\

& $(3;1,2;1,2;0,1,2)$   &  $\langle e_4,e_1,e_1+e_2+e_3\rangle$ \\

& $(3;1,3;1,1;0,1,2)$   &  $\langle e_4,e_2,e_1\rangle$ \\

\hline $2)$ & $(1;0,0;0,0;1,1,1)$   &  $\langle e_3\rangle$ \\

& $(1;0,0;0,1;0,1,1)$   &  $\langle e_2+e_3\rangle$ \\

& $(1;0,1;0,0;0,1,1)$   &  $\langle e_2\rangle$ \\

& $(1;0,1;0,1;0,0,1)$   &  $\langle e_1+e_2\rangle$ \\

& $(1;0,1;1,1;0,0,0)$   &  $\langle e_4\rangle$ \\

& $(1;1,1;0,0;0,0,0)$   &  $\langle e_1+e_4\rangle$ \\

& $(2;0,1;0,1;1,2,2)$   &  $\langle e_3,e_2\rangle$ \\

& $(2;0,1;0,2;0,1,2)$   &  $\langle e_2+e_3,e_1+e_2\rangle$ \\

& $(2;0,1;1,1;1,1,1)$   &  $\langle e_4,e_3\rangle$ \\

& $(2;0,1;1,2;0,1,1)$   &  $\langle e_4,e_2+e_3\rangle$ \\

& $(2;0,2;0,1;0,1,2)$   &  $\langle e_2,e_1\rangle$ \\

& $(2;0,2;1,1;0,1,1)$   &  $\langle e_4,e_2\rangle$ \\

& $(2;0,2;1,2;0,0,1)$   &  $\langle e_4,e_1+e_2\rangle$ \\

& $(2;1,1;0,1;1,1,1)$   &  $\langle e_3,e_1+e_4\rangle$ \\

& $(2;1,2;0,1;0,1,1)$   &  $\langle e_2,e_1+e_4\rangle$ \\

& $(2;1,2;1,1;0,0,1)$   &  $\langle e_4,e_1\rangle$ \\

& $(3;0,2;0,2;1,2,3)$   &  $\langle e_3,e_2,e_1\rangle$ \\

& $(3;0,2;1,2;1,2,2)$   &  $\langle e_4,e_3,e_2\rangle$ \\

& $(3;0,2;1,3;0,1,2)$   &  $\langle e_4,e_2+e_3,e_1+e_2\rangle$ \\

& $(3;1,2;0,2;1,2,2)$   &  $\langle e_3,e_2,e_1+e_4\rangle$ \\

& $(3;1,2;1,2;1,1,2)$   &  $\langle e_4,e_3,e_1\rangle$ \\

& $(3;1,3;1,2;0,1,2)$   &  $\langle e_4,e_2,e_1\rangle$ \\

\hline $3)$ & $(1;0,0;0,1;0,1,1)$   &  $\langle e_2+e_4\rangle$ \\

& $(1;0,1;0,0;1,1,1)$   &  $\langle e_3\rangle$ \\

& $(1;0,1;0,1;0,0,1)$   &  $\langle e_1+e_2+e_3\rangle$ \\

& $(1;0,1;1,1;0,0,0)$   &  $\langle e_5\rangle$ \\

& $(1;1,1;0,0;0,0,1)$   &  $\langle e_1\rangle$ \\

& $(2;0,1;0,0;2,2,2)$   &  $\langle e_4,e_3\rangle$ \\

& $(2;0,1;0,1;1,2,2)$   &  $\langle e_4,e_2\rangle$ \\

& $(2;0,1;0,2;0,1,2)$   &  $\langle e_2+e_4,e_1+e_2+e_3\rangle$ \\

& $(2;0,1;1,2;0,1,1)$   &  $\langle e_5,e_2+e_4\rangle$ \\

& $(2;0,2;0,0;1,2,2)$   &  $\langle e_3,e_2\rangle$ \\

& $(2;0,2;0,1;1,1,2)$   &  $\langle e_3,e_1+e_2+e_3\rangle$ \\

& $(2;0,2;1,1;1,1,1)$   &  $\langle e_5,e_3\rangle$ \\

& $(2;0,2;1,2;0,0,1)$   &  $\langle e_5,e_1+e_2+e_3\rangle$ \\

& $(2;1,1;0,1;1,1,1)$   &  $\langle e_4,e_2+e_5\rangle$ \\

& $(2;1,2;0,0;1,1,2)$   &  $\langle e_3,e_1\rangle$ \\

& $(2;1,2;0,1;0,1,2)$   &  $\langle e_1,e_1+e_2+e_3\rangle$ \\

& $(2;1,2;1,1;0,1,1)$   &  $\langle e_5,e_2\rangle$ \\

& $(2;2,2;0,0;0,0,1)$   &  $\langle e_2+e_5,e_1\rangle$ \\

& $(3;0,2;0,1;2,3,3)$   &  $\langle e_4,e_3,e_2\rangle$ \\

& $(3;0,2;1,1;2,2,2)$   &  $\langle e_5,e_4,e_3\rangle$ \\

& $(3;0,2;1,3;0,1,2)$   &  $\langle e_5,e_2+e_4,e_1+e_2+e_3\rangle$ \\

& $(3;1,2;0,1;2,2,3)$   &  $\langle e_4,e_3,e_1\rangle$ \\

& $(3;1,2;0,2;1,2,3)$   &  $\langle e_2+e_4,e_1,e_1+e_2+e_3\rangle$ \\

& $(3;1,2;1,2;1,2,2)$   &  $\langle e_5,e_4,e_2\rangle$ \\

& $(3;1,3;0,1;1,2,3)$   &  $\langle e_3,e_2,e_1\rangle$ \\

& $(3;1,3;1,1;1,2,2)$   &  $\langle e_5,e_3,e_2\rangle$ \\

& $(3;1,3;1,2;1,1,2)$   &  $\langle e_5,e_3,e_1+e_2+e_3\rangle$ \\

& $(3;2,3;0,1;1,1,2)$   &  $\langle e_3,e_2+e_5,e_1\rangle$ \\

& $(3;2,3;1,1;0,1,2)$   &  $\langle e_5,e_2,e_1\rangle$ \\

& $(4;1,3;0,2;2,3,4)$   &  $\langle e_4,e_3,e_2,e_1\rangle$ \\

& $(4;1,3;1,2;2,3,3)$   &  $\langle e_5,e_4,e_3,e_2\rangle$ \\

& $(4;1,3;1,3;1,2,3)$   &  $\langle e_5,e_4,e_2,e_1+e_2+e_3\rangle$ \\

& $(4;2,3;0,2;2,2,3)$   &  $\langle e_4,e_3,e_2+e_5,e_1\rangle$ \\

& $(4;2,4;1,2;1,2,3)$   &  $\langle e_5,e_3,e_2,e_1\rangle$ \\

\hline $4)$ & $(1;0,0;0,0;1,1,1)$   &  $\langle e_4\rangle$ \\

& $(1;0,0;0,1;0,1,1)$   &  $\langle e_2+e_3+e_4\rangle$ \\

& $(1;0,1;0,0;0,1,1)$   &  $\langle e_3\rangle$ \\

& $(1;0,1;0,1;0,0,1)$   &  $\langle e_1+e_3\rangle$ \\

& $(1;0,1;1,1;0,0,0)$   &  $\langle e_5\rangle$ \\

& $(1;1,1;0,0;0,0,1)$   &  $\langle e_1\rangle$ \\

& $(2;0,1;0,1;1,2,2)$   &  $\langle e_4,e_2+e_3+e_4\rangle$ \\

& $(2;0,1;0,2;0,1,2)$   &  $\langle e_2+e_3+e_4,e_1+e_3\rangle$ \\

& $(2;0,1;1,1;1,1,1)$   &  $\langle e_5,e_4\rangle$ \\

& $(2;0,1;1,2;0,1,1)$   &  $\langle e_5,e_2+e_3+e_4\rangle$ \\

& $(2;0,2;0,0;0,2,2)$   &  $\langle e_3,e_2\rangle$ \\

& $(2;0,2;1,2;0,0,1)$   &  $\langle e_5,e_1+e_3\rangle$ \\

& $(2;1,1;0,0;1,1,2)$   &  $\langle e_4,e_1\rangle$ \\

& $(2;1,2;0,1;0,1,2)$   &  $\langle e_3,e_1\rangle$ \\

& $(2;1,2;1,1;0,1,1)$   &  $\langle e_5,e_2\rangle$ \\

& $(2;2,2;0,0;0,0,1)$   &  $\langle e_2+e_5,e_1\rangle$ \\

& $(3;0,2;0,1;1,3,3)$   &  $\langle e_4,e_3,e_2\rangle$ \\

& $(3;0,2;0,2;1,2,3)$   &  $\langle e_4,e_2+e_3+e_4,e_1+e_3\rangle$ \\

& $(3;0,2;1,2;1,2,2)$   &  $\langle e_5,e_4,e_2+e_3+e_4\rangle$ \\

& $(3;0,2;1,3;0,1,2)$   &  $\langle e_5,e_2+e_3+e_4,e_1+e_3\rangle$ \\

& $(3;1,2;0,1;1,2,3)$   &  $\langle e_4,e_3,e_1\rangle$ \\

& $(3;1,2;0,2;0,2,3)$   &  $\langle e_3,e_2+e_3+e_4,e_1\rangle$ \\

& $(3;1,2;1,1;1,2,2)$   &  $\langle e_5,e_4,e_2\rangle$ \\

& $(3;1,2;1,2;0,2,2)$   &  $\langle e_5,e_2,e_2+e_3+e_4\rangle$ \\

& $(3;1,3;0,1;0,2,3)$   &  $\langle e_3,e_2,e_1\rangle$ \\

& $(3;1,3;1,1;0,2,2)$   &  $\langle e_5,e_3,e_2\rangle$ \\

& $(3;1,3;1,2;0,1,2)$   &  $\langle e_5,e_3,e_1\rangle$ \\

& $(3;2,2;0,1;1,1,2)$   &  $\langle e_4,e_2+e_5,e_1\rangle$ \\

& $(3;2,3;1,1;0,1,2)$   &  $\langle e_5,e_2,e_1\rangle$ \\

& $(4;1,3;0,2;1,3,4)$   &  $\langle e_4,e_3,e_2,e_1\rangle$ \\

& $(4;1,3;1,2;1,3,3)$   &  $\langle e_5,e_4,e_3,e_2\rangle$ \\

& $(4;1,3;1,3;1,2,3)$   &  $\langle e_5,e_4,e_2+e_3+e_4,e_1+e_3\rangle$ \\

& $(4;2,3;1,2;1,2,3)$   &  $\langle e_5,e_4,e_2,e_1\rangle$ \\

& $(4;2,4;1,2;0,2,3)$   &  $\langle e_5,e_3,e_2,e_1\rangle$ \\

\hline $5)$ & $(1;0,0;0,0;1,1,1)$   &  $\langle e_4\rangle$ \\

& $(1;0,1;0,0;0,1,1)$   &  $\langle e_3\rangle$ \\

& $(1;0,1;0,1;0,0,1)$   &  $\langle e_1+e_2+e_3\rangle$ \\

& $(1;0,1;1,1;0,0,0)$   &  $\langle e_5\rangle$ \\

& $(1;1,1;0,0;0,0,1)$   &  $\langle e_1\rangle$ \\

& $(2;0,1;0,0;1,2,2)$   &  $\langle e_4,e_3\rangle$ \\

& $(2;0,1;0,2;0,0,2)$   &  $\langle e_1+e_4,e_1+e_2+e_3\rangle$ \\

& $(2;0,1;1,1;1,1,1)$   &  $\langle e_5,e_4\rangle$ \\

& $(2;0,2;0,1;0,1,2)$   &  $\langle e_3,e_1+e_2+e_3\rangle$ \\

& $(2;0,2;1,1;0,1,1)$   &  $\langle e_5,e_3\rangle$ \\

& $(2;0,2;1,2;0,0,1)$   &  $\langle e_5,e_1+e_2+e_3\rangle$ \\

& $(2;1,1;0,1;1,1,2)$   &  $\langle e_4,e_1\rangle$ \\

& $(2;1,2;0,0;0,1,2)$   &  $\langle e_3,e_1\rangle$ \\

& $(2;1,2;0,1;0,0,2)$   &  $\langle e_1,e_1+e_2+e_3\rangle$ \\

& $(2;1,2;1,1;0,0,1)$   &  $\langle e_5,e_2\rangle$ \\

& $(2;2,2;0,0;0,0,1)$   &  $\langle e_2+e_5,e_1\rangle$ \\

& $(3;0,2;1,1;1,2,2)$   &  $\langle e_5,e_4,e_3\rangle$ \\

& $(3;0,2;1,3;0,0,2)$   &  $\langle e_5,e_1+e_4,e_1+e_2+e_3\rangle$ \\

& $(3;1,2;0,1;1,2,3)$   &  $\langle e_4,e_3,e_1\rangle$ \\

& $(3;1,2;0,2;1,1,3)$   &  $\langle e_4,e_1,e_1+e_2+e_3\rangle$ \\

& $(3;1,2;1,2;1,1,2)$   &  $\langle e_5,e_4,e_1\rangle$ \\

& $(3;1,3;0,1;0,1,3)$   &  $\langle e_3,e_2,e_1\rangle$ \\

& $(3;1,3;1,2;0,1,2)$   &  $\langle e_5,e_3,e_1+e_2+e_3\rangle$ \\

& $(3;2,2;0,1;1,1,2)$   &  $\langle e_4,e_2+e_5,e_1\rangle$ \\

& $(3;2,3;0,1;0,1,2)$   &  $\langle e_3,e_2+e_5,e_1\rangle$ \\

& $(3;2,3;1,1;0,0,2)$   &  $\langle e_5,e_2,e_1\rangle$ \\

& $(4;1,3;0,2;1,2,4)$   &  $\langle e_4,e_3,e_2,e_1\rangle$ \\

& $(4;1,3;1,2;1,2,3)$   &  $\langle e_5,e_4,e_3,e_2\rangle$ \\

& $(4;1,3;1,3;1,1,3)$   &  $\langle e_5,e_4,e_1,e_1+e_2+e_3\rangle$ \\

& $(4;2,3;0,2;1,2,3)$   &  $\langle e_4,e_3,e_2+e_5,e_1\rangle$ \\

& $(4;2,3;1,2;1,1,3)$   &  $\langle e_5,e_4,e_2,e_1\rangle$ \\

& $(4;2,4;1,2;0,1,3)$   &  $\langle e_5,e_3,e_2,e_1\rangle$ \\

\hline $6)$ & $(1;0,0;0,0;1,1,1)$   &  $\langle e_3\rangle$ \\

& $(1;0,0;0,1;0,0,1)$   &  $\langle e_1+e_2+e_3\rangle$ \\

& $(1;0,1;0,0;0,1,1)$   &  $\langle e_2\rangle$ \\

& $(1;0,1;1,1;0,0,0)$   &  $\langle e_5\rangle$ \\

& $(1;1,1;0,0;0,0,0)$   &  $\langle e_4\rangle$ \\

& $(2;0,1;0,0;1,2,2)$   &  $\langle e_3,e_2\rangle$ \\

& $(2;0,1;0,1;1,1,2)$   &  $\langle e_3,e_1+e_2+e_3\rangle$ \\

& $(2;0,1;1,1;1,1,1)$   &  $\langle e_5,e_3\rangle$ \\

& $(2;0,1;1,2;0,0,1)$   &  $\langle e_5,e_1+e_2+e_3\rangle$ \\

& $(2;0,2;0,0;0,1,2)$   &  $\langle e_2,e_1\rangle$ \\

& $(2;0,2;1,1;0,1,1)$   &  $\langle e_5,e_2\rangle$ \\

& $(2;0,2;1,2;0,0,0)$   &  $\langle e_5,e_2+e_4\rangle$ \\

& $(2;1,1;0,0;1,1,1)$   &  $\langle e_4,e_3\rangle$ \\

& $(2;1,2;0,1;0,1,1)$   &  $\langle e_4,e_2\rangle$ \\

& $(2;1,2;1,1;0,0,1)$   &  $\langle e_5,e_1\rangle$ \\

& $(2;2,2;0,0;0,0,0)$   &  $\langle e_4,e_1+e_5\rangle$ \\

& $(3;0,2;0,1;1,2,3)$   &  $\langle e_3,e_2,e_1\rangle$ \\

& $(3;0,2;1,1;1,2,2)$   &  $\langle e_5,e_3,e_2\rangle$ \\

& $(3;0,2;1,2;1,1,2)$   &  $\langle e_5,e_3,e_1+e_2+e_3\rangle$ \\

& $(3;0,2;1,3;0,0,1)$   &  $\langle e_5,e_2+e_4,e_1+e_2+e_3\rangle$ \\

& $(3;1,2;0,1;1,2,2)$   &  $\langle e_4,e_3,e_2\rangle$ \\

& $(3;1,2;1,1;1,1,2)$   &  $\langle e_5,e_3,e_1\rangle$ \\

& $(3;1,2;1,2;0,1,2)$   &  $\langle e_5,e_1,e_1+e_2+e_3\rangle$ \\

& $(3;1,3;1,1;0,1,2)$   &  $\langle e_5,e_2,e_1\rangle$ \\

& $(3;1,3;1,2;0,1,1)$   &  $\langle e_5,e_4,e_2\rangle$ \\

& $(3;2,2;0,1;1,1,1)$   &  $\langle e_4,e_3,e_1+e_5\rangle$ \\

& $(3;2,3;0,1;0,1,1)$   &  $\langle e_4,e_2,e_1+e_5\rangle$ \\

& $(3;2,3;1,1;0,0,1)$   &  $\langle e_5,e_4,e_1\rangle$ \\

& $(4;1,3;1,2;1,2,3)$   &  $\langle e_5,e_3,e_2,e_1\rangle$ \\

& $(4;1,3;1,3;1,1,2)$   &  $\langle e_5,e_3,e_2+e_4,e_1+e_2+e_3\rangle$ \\

& $(4;2,3;0,2;1,2,2)$   &  $\langle e_4,e_3,e_2,e_1+e_5\rangle$ \\

& $(4;2,3;1,2;1,1,2)$   &  $\langle e_5,e_4,e_3,e_1\rangle$ \\

& $(4;2,4;1,2;0,1,2)$   &  $\langle e_5,e_4,e_2,e_1\rangle$ \\

\end{longtable}

\end{lem}


\begin{thebibliography}{99}


\bibitem{Dolgachev}
{ Dolgachev Igor} \textit{Lectures on invariant theory.} {Cambridge
University Press} (2002).



\bibitem{Drozd}
{ Drozd Ju. A. } \textit{Coxeter transformations and representations
of partially ordered sets.} {Funkcional. Anal. i Prilozen.}
(Russian) \textbf{8} (1974), {no.~3,} {P.34--42}.


\bibitem{GabrielRoiter}
{ Gabriel P., Roiter A. V.} \textit{ Representations of
Finite-Dimensional Algebras.} {Enc. of Math. Sci.}, \textbf{73},
Algebra VIII, (Springer, 1992).

\bibitem{GrushevoyYusenko} Grushevoy Roman, Yusenko Kostyantyn, {\em On the unitarization of
linear representations of primitive partially ordered sets.} to
appear in Operator theory and Application, Birkhauser, 2009.


\bibitem{Hu2004} Hu Yi, {\em Stable configurations of linear subspaces and quotient coherent
sheaves}. Pure and Applied Mathamtics Quartely, \textbf{1} (2005),
{no.~1,} {P.127--164}.


\bibitem{Kleiner1}
{Kleiner M. M.} \textit{ Partially ordered sets of finite type.}
 {Investigations on the theory of representations. Zap.
Nau\v cn. Sem. Leningrad. Otdel. Mat. Inst. Steklov.
(LOMI)}(Russian) \textbf{28} (1972), {P.32--41}.

\bibitem{Kleiner2}
{Kleiner M. M.} \textit{ On the exact represenattions of partially
ordered sets of finite type.}
 {Investigations on the theory of representations. Zap.
Nau\v cn. Sem. Leningrad. Otdel. Mat. Inst. Steklov.
(LOMI)}(Russian) \textbf{28} (1972), {P.42--59}.

\bibitem{Klyachko}
{Klyachko A.A.} \textit{Stable bundles representation theory and
Hermitian operatos.}
 {Selecta Math.} \textbf{4} (1998), {P.419--445}.


\bibitem{KruglyaSamoilenko2}
{ Krugljak S. A., Samoilenko Yu. S.}   
\textit{On the complexity of description of representations
 of $*$-algebras generated by idempotents.} 
{Proceedings of the American Mathematical Society} 
 \textbf{128}, (2000),
{no.~6},   
{P.1655--1664}. 



\bibitem{KruglyaNazarovaRoiter}
{ Krugljak S. A., Nazarova L. A., Roiter A. V.}   
\textit{Orthoscalar representations of quivers on the category of Hilberts spaces. II.} 
accepted for publication in Funct. Anal. and Application

\bibitem{Nazarova}
{Nazarova L. A.} \textit{Partially ordered sets of infinite type.}
{Izv. Akad. Nauk SSSR}(Russian) \textbf{39} (1975), {no.~5},
{P.963--991}.



\bibitem{NazarovaRoiter}
{Nazarova L. A., Roiter A. V.} \textit{Representations of the
partially ordered sets.} {Investigations on the theory of
representations. Zap. Nau\v cn. Sem. Leningrad. Otdel. Mat. Inst.
Steklov. (LOMI)}(Russian) \textbf{28} (1972), {P.5--31}.


\bibitem{OstrovskyiSamoilenko}
{Ostrovskyi V. L., Samoilenko Yu. S.} \textit{On spectral theorems
for families of linearly connected selfadjoint operators with
prescribed spectra associated with extended Dynkin graphs} {Ukrain.
Mat. Zh.} (Ukrainian) \textbf{58} (2006), {no.~11}, {P.1556--1570}.


\bibitem{PadbergManfred}
Padberg, Manfred  \textit{Linear Optimization and Extensions.}
Springer, Berin  (1995).


\bibitem{PopovichTurowska}
Popovich Stanislav, Turowska Lyudmila \textit{Unitarizable and
non-unitarizable representations of algebras generated by
idempotents.}  Algebra and Discrete Mathematics \textbf{3} (2004),
111--125.

\bibitem{Simson1992} Simson Daniel, \emph{Linear representations of partially
ordered sets and vector space categories.} Gordon and Breach Science
Publishers, 1992.

\bibitem{Totaro1994} Totaro B., \emph{Tensor products of semistables are
semistable.} Geometry and Analysis on Complex Manifolds, World
Series, River Edge, N.J., 1994.

\bibitem{Wanf2002} Wang Xiaowei, \emph{Balance point and stability of vector bundles over projective manifold.}
Mathematicl Research Letters \textbf{9} (2002), 393--411.

\bibitem{Ya2010} Yakimenko D. Yu., \emph{On unitarization of schurian representations of partially ordered set $\widetilde {E_7}$.}
accepted for publication in {Ukrain. Mat. Zh.}

\end{thebibliography}
\end{document}